\documentclass[11pt,reqno]{amsart}
\usepackage{hyperref}
\usepackage[latin1]{inputenc}
\usepackage{amssymb}
\usepackage{amsmath}
\usepackage{esint}
\usepackage{tikz, pgf, xcolor, graphicx}

\newcommand{\rn}{{\mathbb{R}^n}}
\newcommand{\phii}{\varphi}
\newcommand{\pp}{\partial}
\newcommand{\w}{\omega}
\newcommand{\ws}{\omega^s_\Omega}

\newcommand{\W}{\Omega}
\newcommand{\eps}{\varepsilon}
\newcommand{\regf}{-s+1/2}
\newcommand{\regg}{s+1/2}
\def\new{}

\def\R{{\mathbb {R}}}

\def\T{{\mathcal {T}}}

\newtheorem{theorem}{Theorem}[section]
\newtheorem{lemma}[theorem]{Lemma}
\newtheorem{proposition}[theorem]{Proposition}
\newtheorem{corollary}[theorem]{Corollary}

\theoremstyle{remark}
\newtheorem{remark}[theorem]{Remark}

\theoremstyle{definition}
\newtheorem{definition}[theorem]{Definition}

\numberwithin{equation}{section}

\title[FE for the nonhomogeneous fractional Dirichlet problem]{Finite element approximations of the nonhomogeneous fractional Dirichlet problem}

\thanks{Supported by CONICYT-Chile through Fondecyt project 1150056, by CONICET under grant PIP 2014-2016
11220130100184CO and by ANPCyT under grant 2014-1771}

\author[G. Acosta, J.P. Borthagaray and N. Heuer]{Gabriel Acosta, Juan Pablo Borthagaray and Norbert Heuer}

\address[G. Acosta]{IMAS - CONICET and Departamento de Matem\'a\-tica, FCEyN - Universidad de Buenos Aires, Ciudad Universitaria, Pabell\'on I  (1428) Buenos Aires, Argentina.}

\address[J.P. Borthagaray]{IMAS - CONICET and Departamento de Matem\'a\-tica, FCEyN - Universidad de Buenos Aires, Ciudad Universitaria, Pabell\'on I  (1428) Buenos Aires, Argentina and Department of Mathematics, University of Maryland, College Park, MD 20742, USA.}

\address[N. Heuer]{Facultad de Matem\'aticas, Pontificia Universidad Cat\'olica de Chile, Avenida Vicku\~na Mackenna 4860, Santiago, Chile.}

\email[G. Acosta]{gacosta@dm.uba.ar}

\email[J.P. Borthagaray]{jpb@umd.edu}

\email[N. Heuer]{nheuer@mat.uc.cl}


\begin{document}
\begin{abstract}
We study finite element approximations of the nonhomogeneous Dirichlet problem for the fractional Laplacian. Our approach is based on weak imposition of the Dirichlet condition and incorporating a nonlocal analogous of the normal derivative as a Lagrange multiplier in the formulation of the problem. In order to obtain convergence orders for our scheme, regularity estimates are developed, both for the solution and its nonlocal derivative. The method we propose requires that, as meshes are refined, the discrete problems be solved in a family of domains of growing diameter.
\end{abstract}

\maketitle

\section{Introduction and preliminaries} \label{sec:intro}
  Anomalous diffusion refers to phenomena arising whenever the associated underlying stochastic process is not given by Brownian motion. 
One striking example of a nonlocal operator is the fractional Laplacian of order $s$ ($0<s<1$), which we will denote by $(-\Delta)^s$.

If the domain under consideration is the whole space $\rn$, then $(-\Delta)^s$ is a pseudodifferential operator with symbol $|\xi|^{2s}$. Indeed, for a function $u$ in the Schwartz class $\mathcal{S}$, let
\begin{equation}
(-\Delta)^s u = \mathcal{F}^{-1} \left( |\xi|^{2s} \mathcal{F} u \right) ,
\label{eq:fourier}
\end{equation}
where $\mathcal{F}$ denotes the Fourier transform. 
The fractional Laplacian can equivalently be defined by means of the identity \cite{Hitchhikers}
\begin{equation}
(-\Delta)^s u (x) = C(n,s) \mbox{ P.V.} \int_\rn \frac{u(x)-u(y)}{|x-y|^{n+2s}} \, dy,
\label{eq:fraccionarioyo}
\end{equation}
where the normalization constant 
\begin{equation} \label{eq:cns}
C(n,s) = \frac{2^{2s} s \Gamma(s+\frac{n}{2})}{\pi^{n/2} \Gamma(1-s)} 
\end{equation}
is taken in order to be consistent with definition \eqref{eq:fourier}.

In the theory of stochastic processes, this operator appears as the infinitesimal generator of a stable Lévy process \cite{Bertoin}. Indeed, it 
is possible to obtain a fractional heat equation as a limit of a random walk with long jumps \cite{Valdinoci}. 

There are two different approaches to the definition of the fractional Laplacian on an open bounded set $\W$. 
On the one hand, to analyze powers of the Laplacian in a spectral sense: given a function $u$, to consider its spectral decomposition in terms of the eigenfunctions of the Laplacian with homogeneous Dirichlet boundary condition, and to take the operator that acts by raising to the power $s$ the corresponding eigenvalues. Namely, if 
$\{\psi_k, \lambda_k \}_{k \in \mathbb{N}} \subset H^1_0 (\W) \times \mathbb{R}_+ $ denotes the set of normalized eigenfunctions and eigenvalues, then this operator is defined as 
\[
(-\Delta)_S^s \, u (x) =  \sum_{k=1}^\infty \lambda_k^s ( u , \psi_k )_{L^2(\W)} \psi_k(x),
   \qquad x\in\Omega.
\]

On the other hand, there is the possibility to keep the motivation coming from the stochastic process leading to the definition of $(-\Delta)^s$ in $\rn$. 
This option leads to two different types of operators: one in which the stochastic process is restricted to $\W$ and one in which particles are allowed to jump anywhere in the space.
The first of these two is the infinitesimal generator of a censored stable Lévy process \cite{Bogdan}, we refer to it as \emph{regional} fractional Laplacian and it is given by 
\begin{equation}
(-\Delta)^s_{\W} u(x) = C(n,s,\W) \mbox{ P.V.} \int_\W \frac{u(x)-u(y)}{|x-y|^{n+2s}} \, dy,
   \quad x\in\Omega.
\label{eq:regional}
\end{equation}

The second of the two operators  motivated  by L\'evy processes leads to considering the integral formulation \eqref{eq:fraccionarioyo}. Observe that, unlike the aforementioned fractional Laplacians, the definition of this operator does not depend on the domain $\W$. In this work we deal with this operator, which we denote by $(-\Delta)^s$ and simply call it the fractional Laplacian. 
The possibility of having arbitrarily long jumps in the random walk explains why, when considering a fractional Laplace equation on
a bounded domain $\W$, boundary conditions should be prescribed on $\W^c = \rn \setminus \overline \W$. 

For an account of numerical methods for the fractional Laplacians mentioned above, we refer the reader to the recent survey \cite{survey}.
Specific to the numerical treatment of \eqref{eq:fraccionarioyo}, we mention algorithms based on finite elements \cite{ABB, AB, AinsworthGlusa_adaptive, AinsworthGlusa_efficient, DEliaGunzburger}, 
finite differences \cite{HuangOberman}, Dunford-Taylor representation formulas \cite{BLP3}, Nystr\"om \cite{ABBM} and Monte Carlo \cite{Kyprianou} methods.

\new{Given $s \in (0,1)$}, in this work we study finite element approximations to problem
\begin{equation} \label{eq:dirichletnh}
\left\lbrace
  \begin{array}{rl}
      (-\Delta)^s u = f & \mbox{ in }\W, \\
      u =  g & \mbox{ in }\W^c , \\
      \end{array}
    \right.
\end{equation}
where the functions $f$ and $g$ are data belonging to suitable spaces. Analysis of the homogeneous counterpart of \eqref{eq:dirichletnh} was carried out in \cite{AB}, 
where a numerical method was developed, theoretical error bounds were established and numerical results in agreement with the theoretical predictions were obtained. 
Solvability of a class of nonhomogeneous Dirichlet problems for nonlocal operators --involving not necessarily symmetric or continuous kernels--
was studied in \cite{Felsinger2015}.

An important result for dealing with \eqref{eq:dirichletnh} is the following integration 
by parts formula for the fractional Laplacian  \cite{dipierro2014nonlocal}: for $u, v$ smooth enough,
it holds
\begin{equation}\label{eq:parts} \begin{split}
\frac{C(n,s)}{2} \iint_Q &  \frac{(u(x)-u(y))  (v(x)-v(y))}{|x-y|^{n+2s}} \, dx \, dy \,  \\ 
& = \int_\W  v(x) (-\Delta)^su(x) \, dx + \int_{\W^c} v(x) \,  \mathcal{N}_s u(x) \, dx ,
\end{split} \end{equation}
where $\mathcal{N}_s u$ is the \emph{nonlocal normal derivative} of $u$, given by 
 \begin{equation*} 
\mathcal{N}_s u = C(n,s) \, \int_\W \frac{u(x)-u(y)}{|x-y|^{n+2s}} \, dy, \ x \in \W^c,
\end{equation*}
and $Q = (\W \times \rn) \cup (\rn \times \W)$. 
\new{Along this paper we always work with a \emph{fixed} value of $s$. Nonetheless, it is instructive to mention that $\mathcal{N}_s u$ recovers in the limit $s\to 1$ the notion of the classical normal derivative (cf. Remark \ref{rem:singular}).}

The aim of this work is to build finite element approximations for both, the solution $u$ of \eqref{eq:dirichletnh} as well as for its nonlocal derivative $\mathcal{N}_s u$.  In this regard, we discuss briefly a standard direct approach in which the Dirichlet condition $g$ is strongly imposed. As it turns out, this simple and optimally convergent method for the variable $u$, does not provide a computable approximation of $\mathcal{N}_s u$. 
In order to overcome this limitation  a mixed formulation of the problem --in which $\mathcal{N}_s u$ plays the role of a Lagrange multiplier--
 is introduced and numerically approximated. By means of this approach, which is the main object of this paper,
 numerical approximations for both $u$ and $\mathcal{N}_s u$ are delivered and optimal order of convergence is proved for them.
\new{In this way,  our method inaugurates the variational setting
for the treatment of non-homogeneous  essential boundary conditions of fractional operators. This is a promising  scenario in which one might consider more general problems, including coupled systems involving fractional and integer-order operators.}

Throughout this paper, $C$ denotes a positive constant which may be different in various places.

 \subsection{Sobolev spaces}
Given an open set $\W \subset \rn$ and $s \in(0,1)$,  the fractional Sobolev space $H^s(\W)$ is defined by
\[
H^s(\W) = \left\{ v \in L^2(\W) \colon |v|_{H^s(\W)} < \infty \right\},
\]
where $|\cdot|_{H^s(\W)}$ is the Aronszajn-Slobodeckij seminorm
\[
|v|_{H^s(\W)}^2 =  \iint_{\W^2} \frac{|v(x)-v(y)|^2}{|x-y|^{n+2s}} \, dx \, dy.
\]

Naturally, $H^s(\W)$ is a Hilbert space furnished with the norm $\|\cdot\|_{H^s(\W)}^2 = \|\cdot\|_{L^2(\W)}^2 + |\cdot|_{H^s(\W)}^2 .$
We denote $\langle \cdot, \cdot \rangle_{H^s(\W)}$ the bilinear form
$$
\langle u , v \rangle_{H^s(\W)} = \iint_{\W^2} \frac{(u(x)-u(y))(v(x)-v(y))}{|x-y|^{n+2s}} \, dx \, dy, \quad u, v \in H^s(\W). 
 $$

Sobolev spaces of order greater than one are defined as follows. If $s>1$ is not an integer, the decomposition $s = m + \sigma$, where $m \in \mathbb{N}$ and $\sigma \in (0,1)$,  allows to define $H^s(\W)$ by setting
\[
H^s(\W) = \left\{ v \in H^m(\W) \colon |D^\alpha v|_{H^{\sigma}(\W)} < \infty \text{ for all } \alpha \text{ s.t. }  |\alpha| = m \right\}.
\]

A space of interest in our analysis consists of the set
$$\widetilde{H}^s(\W) = \{ v \in H^s(\rn) \colon \text{supp } v \subset \overline{\W} \},$$
endowed with the norm 
\[
\| v \|_{\widetilde H^s(\W)} = \| \tilde v \|_{H^s(\rn)},
\]
where $\tilde v$ is the extension of $v$ by zero outside $\W$. For simplicity
of notation, whenever we refer to a function in $\widetilde{H}^s(\W)$, we assume that it is extended by zero onto $\W^c$.

Let $s > 0$. By using $L^2(\W)$ as a pivot space, we have that the duality pairing between $H^s(\W)$ and its dual $\widetilde{H}^{-s}(\W) = (H^s(\W))'$ coincides with the $L^2(\W)$ inner product. Moreover, we denote the dual of $\widetilde{H}^{s}(\W)$ by $H^{-s}(\W)$. 

\begin{remark}[Duality pairs]
In order to keep the notation as clear as possible, along the following sections 
we    write $\int_\W \mu v$  for $\mu\in H'$ and $v\in H$. However, if the duality 
needs to be stressed we use $\langle \mu, v\rangle$ instead.
\label{rem:dualitypair}
\end{remark}

We state some important theoretical results regarding the space $\widetilde{H}^{s}(\W)$ (see e.g.,  \cite[Proposition 2.4]{AB}). 
\begin{proposition}[Poincar\'e inequality] Given a domain $\W$ and $s>0$, there exists a constant $C$ such that, for all $v \in \widetilde H^s(\W)$,
\begin{equation}
\| v \|_{L^2(\W)} \le C | v |_{H^s(\rn)}.
\label{eq:poincare}
\end{equation}
\end{proposition}

\begin{remark} \label{rem:poincare}
Analogously to integer order Sobolev spaces, an immediate consequence of the Poincar\'e inequality is that the $H^s$-seminorm is equivalent to the full $H^s$-norm over $\widetilde H ^s(\W)$. Observe that, given $v \in \widetilde{H}^s (\W)$, its $H^s$-seminorm is given by
\[
|v|_{H^s(\rn)}^2 = |v|_{H^s(\W)}^2 + 2 \int_\W |v(x)|^2 \int_{\W^c} \frac{1}{|x-y|^{n+2s}} dy \, dx .
\]
\end{remark}

\begin{definition} \label{def:w}
Given a (not necessarily bounded) set $\W$ with Lipschitz continuous boundary and $s \in (0,1)$, we denote by $\ws : \W \to (0,\infty)$ the function given by
\begin{equation} \label{eq:def_w}
\ws (x) = \int_{\W^c} \frac{1}{|x-y|^{n+2s}} dy.
\end{equation}
Denoting $\delta(x) = d(x,\partial \W)$, the following bounds hold 
\[
0 < \frac{C}{\delta(x)^{2s}} \le \ws(x) \le \frac{\sigma_{n-1}}{2s \, \delta(x)^{2s}} \quad \forall x\in \W, 
\]
where $\sigma_{n-1}$ is the measure of the $n-1$ dimensional sphere and $C>0$ depends on $\W$. For the lower bound above we refer to \cite[formula (1.3.2.12)]{Grisvard}, whereas the upper bound is easily deduced by integration in polar coordinates.

\end{definition}

\begin{proposition}[Hardy inequalities, see {\cite{Dyda, Grisvard}}] \label{prop:hardy}
Let $\W$ be a bounded Lipschitz domain, then there exists $c=c(\W,n,s)>0$ such that 
\begin{equation}  \label{eq:hardy}
\begin{split}
\int_\W \frac{|v(x)|^2}{\delta(x)^{2s}} \, dx \, & \leq c \| v \|_{H^s(\W)}^2 \ \forall \, v \in H^s(\W) \quad \text{if } 0<s < 1/2, \\
\int_\W \frac{|v(x)|^2}{\delta(x)^{2s}} \, dx \, & \leq c |v|_{H^s(\W)}^2 \ \forall \, v \in \widetilde H ^s(\W) \quad \text{if } 1/2 < s < 1.
\end{split}
\end{equation}
\end{proposition}

\begin{corollary} \label{cor:norma} If $0<s < 1/2$, then there exists a constant $c=c(\W,n,s)>0$ such that 
\begin{equation*}
\| v \|_{H^s(\rn)} \leq {c} \| v \|_{H^s(\W)}  \quad \forall v \in \widetilde{H}^s(\W).
 \end{equation*}
On the other hand, if $1/2<s<1$ there exists a constant ${c}={c}(\W,n,s)>0$ such that  
\begin{equation*}
\| v \|_{H^s(\rn)} \leq {c} | v |_{H^s({\W})}  \quad \forall v \in \widetilde{H}^s(\W) .
\end{equation*}
\end{corollary}
\begin{remark} \label{remark:un_medio}
When $s=1/2$, since Hardy's inequality fails, it is not possible to bound the $H^{1/2}(\rn)$-seminorm in terms of the $H^{1/2}(\Omega)$-norm for functions supported in $\overline \Omega$. However, for the purposes we pursue in this work, it suffices to notice that the estimate
\[
\|v \|_{H^{1/2}(\rn)} \le C |v|_{H^{1/2+\eps}(\Omega)} 
\]
holds for all $v \in \widetilde {H}^{1/2+\eps}(\Omega)$, 
where  $\eps > 0$ is fixed.
\end{remark}

An important tool for our work is the extension operator given by the following (see \cite[Theorem 5.4]{Hitchhikers} and \cite{zhou2015fractional}).
\begin{lemma}\label{extension}
Given $\sigma \ge 0$ and $\W$ a (not necessarily bounded) Lipschitz domain, there exists a continuous extension operator $E: H^\sigma(\W) \to H^\sigma(\rn).$ Namely, there is a constant $C(n,\sigma,\W)$ such that, for all $u \in H^\sigma(\W)$,
\[
\| Eu \|_{H^\sigma(\rn)} \le C \| u \|_{H^\sigma(\W)}.
\]
\end{lemma}
\begin{remark}
\label{rem:ext_complemento}
During the next sections we need Lemma \ref{extension} for $\W^c$, although we prefer to state it in the more natural fashion, that is, in terms of $\W$ itself.
\end{remark}

\subsection{Fractional Laplacian and regularity of the Dirichlet homogenous problem} 
The operator $(-\Delta)^s$ may be defined either by \eqref{eq:fourier} or \eqref{eq:fraccionarioyo}. The latter is useful to cope with problems involving the operator in a variational framework, and therefore to perform finite element analysis of such problems. On the other hand, definition \eqref{eq:fourier} allows to study the operator from the viewpoint of pseudodifferential calculus. The equivalence between these two definitions can be found, for example, in \cite{Hitchhikers}.
Using the definition \eqref{eq:fourier}, it is easy to prove the following.
\begin{proposition} \label{prop:order} 
For any $s \in \R$, the operator $(-\Delta)^s$ is of order $2s$, that is, $(-\Delta)^s : H^\ell (\rn) \to  H^{\ell-2s} (\rn)$ is continuous for any $\ell \in \R$.
\end{proposition}

 From the previous proposition, it might be expected that, given a bounded smooth domain $\W$, if $u \in \widetilde H^s(\W)$ satisfies $(-\Delta)^s u = f$ for some $f \in H^\ell (\W)$, then $u \in H^{\ell + 2s}(\W)$. However, this is not the case. Regularity of solutions of problems involving the fractional Laplacian over bounded domains is a delicate issue.
Indeed, consider for instance the homogeneous problem
\begin{equation} \label{eq:homogeneous}
\left\lbrace
  \begin{array}{rl}      
(-\Delta)^s u  = f &  \text{ in } \W, \\
u  =  0  & \text{ in } \W^c.
\end{array}
    \right.
\end{equation}
In \cite{Grubb}, regularity results for \eqref{eq:homogeneous} are stated in terms of 
H\"ormander $\mu-$spaces. These mix the features of supported and restricted Sobolev spaces by means of combining certain pseudodifferential operators with zero-extensions and restriction operators.  We refer to that work for a definition and further details. 
In terms of standard Sobolev spaces, the results therein may be stated as follows (see also \cite{VishikEskin}). 
\begin{proposition}\label{prop:regHr}
		Let $f\in H^r(\Omega)$ for $r\geq -s$ and $u\in \widetilde{H}^s(\W)$ be the solution of the Dirichlet problem
		\eqref{eq:homogeneous}. Then, the following regularity estimate holds 
		$$
			|u|_{H^{s+\alpha}(\rn)} \leq C(n, \alpha) \|f\|_{H^r(\W)}.
		$$
		Here, $\alpha = s+r$ if $s+r < \frac12$ and $\alpha = 
		\frac12 - \eps$ if $s+r \ge \frac12$, with $\eps > 0$ arbitrarily small.
\end{proposition}

\begin{remark}
We emphasize that assuming further Sobolev regularity for the right hand side function $f$ does not imply that the solution $u$ will be any smoother than what is given by the previous proposition. 
\end{remark}

\section{Statement of the problem}
Throughout the remaining sections of this work we are going to denote by $V$ the space $V = H^s(\rn)$, 
furnished with its usual norm. The domain $\W$ is assumed to be bounded and smooth and therefore
(due to the latter condition) it is an extension domain for functions
in $H^s(\W)$ (and of course for functions in $H^s(\W^c)$). This fact is used in some parts of the presentation
without further comments. 

Multiplying the first equation in \eqref{eq:dirichletnh} by a suitable test function $v$ and applying \eqref{eq:parts}, we obtain
\begin{equation}
\begin{split}
\frac{C(n,s)}{2} \iint_Q \frac{(u(x)-u(y))(v(x)-v(y))}{|x-y|^{n+2s}} \, dx \, dy & -  \int_{\W^c} v(x) \,  \mathcal{N}_s u(x) \, dx\\
&  = \int_\W f(x) v(x) \, dx .
\end{split}
\label{eq:debil} 
\end{equation}
In order to write a weak formulation for our problem we assume $f\in \widetilde{H}^{-s}(\W)$, $g \in H^s(\W^c)$
and introduce the bilinear and linear forms $a\colon V\times V \to \R,$ $ F\colon V \to \R,$
\[
\begin{aligned}
 a(u,v) & =  \frac{C(n,s)}{2} \iint_Q \frac{(u(x)-u(y))(v(x)-v(y))}{|x-y|^{n+2s}} \, dx \, dy , \\
F(u) & = \int_{\W} f(x) u(x) \, dx ,
\end{aligned} \]
which are needed in the sequel.

\begin{remark} \label{remark:form_a}
The form $a$ satisfies the identity
$$a(u,v) = \frac{C(n,s)}{2} \left( \langle u,v \rangle_{H^s(\rn)} - \langle u,v \rangle_{H^s(\W^c)} \right) \ \forall u, v \in V.$$ 
This, in turn, implies the continuity of $a$ in $V$, that is
 $$|a(u,v)| \le C(n,s) |u|_{H^s(\rn)} |v|_{H^s(\rn)},$$
and the fact that over the set $\widetilde{H}^s(\W)$, $a(v,v)$ coincides with $ \frac{C(n,s)}{2}| v |_{H^s(\rn)}^2.$
\end{remark}

\subsection{Direct formulation} \label{ss:direct}
Our first approach is based on the strong imposition of the Dirichlet condition.
 From \eqref{eq:debil} we obtain at once the weak formulation: find 
$u\in V_g$ such that
\begin{equation}
 \label{eq:cont_direct}
 a(u,v)=F(v) \quad \forall v\in \widetilde H^s(\W),
\end{equation}
where $V_g=\{w\in V \colon w=g \mbox{ in } \W^c\}$. 

The treatment for this formulation is standard. Since $\W^c$ is an extension domain we 
may find $g^E\in V_g$,  $g^E:=E(g)$, such that $\|g^E\|_{V}\le C \|g\|_{H^s(\W^c)}$,
with $C$ depending on $\W$. Using that $a(u,v)$ is continuous and coercive in 
$\widetilde H^s(\W)$ (see Remark \ref{remark:form_a}), existence and uniqueness of a solution
$u_0\in \widetilde H^s(\W)$ of the problem
\[
 a(u_0,v)=F(v)-a(g^E,v) \quad \forall v\in \widetilde H^s(\W),
\]
is guaranteed, thanks to the continuity of the right hand side. Considering $u:=u_0+g^E$ we deduce the following.

\begin{proposition} \label{prop:well_posedness_direct}
Problem \eqref{eq:cont_direct} admits a unique solution $u \in V_g$, and there exists $C>0$ such that the bound
\[
\| u \|_V  \le C \left( \| f \|_{\widetilde{H}^{-s}(\W)} + \| g \|_{H^s(\W^c)} \right)
\]
is satisfied.
\end{proposition}

\subsection{Mixed formulation}
The idea behind this formulation dates back to
Babu\v ska's  seminal paper \cite{Babuska}.  
We  define the set $\Lambda = (H^s(\W^c))' = \widetilde{H}^{-s}(\W^c)$, furnished with its usual norm,
and introduce the bilinear and linear forms $ b\colon V\times \Lambda \to \R, G: \Lambda \to \R$,
\begin{equation*}
 b(u, \mu) = \int_{\W^c} u(x) \,  \mu(x) \, dx ,
\end{equation*}
and 
\begin{align*}
 \ G(\lambda) = \int_{\W^c} g(x) \, \lambda(x) \, dx,
\end{align*} 
which are obviously  continuous.

The mixed formulation of \eqref{eq:dirichletnh} reads: find $(u,\lambda) \in V\times\Lambda$ such that
\begin{equation}\label{eq:cont}
\begin{split}
a(u,v) - b(v,\lambda) & = F(v) \quad \forall v \in V , \\
 b(u,\mu) & = G(\mu) \quad \forall \mu \in \Lambda .
\end{split}
\end{equation}

\begin{remark} \label{rem:lambda}
As can be seen from the above considerations, the Lagrange multiplier $\lambda$, which is associated to the restriction $u = g $ in $\W^c$, 
 coincides with the nonlocal derivative $\mathcal{N}_s u$ in that set. In order to simplify the notation, in the following we will refer to it as $\lambda$.
\end{remark}

Notice that the kernel of the bilinear form $b$ agrees with $\widetilde{H}^s(\W)$,
that is,
\begin{equation} \label{eq:kernel}
K = \{ v \in V \colon b(v,\mu) = 0 \ \forall \mu \in \Lambda \} = \widetilde{H}^s(\W).
\end{equation}
Recalling Remarks \ref{rem:poincare} and \ref{remark:form_a}, it follows that 
\begin{equation}
 \| v \|_V^2 \le C | v |_{H^s(\rn)}^2 = 
C a(v,v)  \quad \forall v \in K.
\label{eq:ellipticity}
\end{equation}

We are now in a position to prove the inf-sup condition for the form $b$.

\begin{lemma}
For all $\mu \in \Lambda$, it holds that
\begin{equation}\label{eq:infsup}
\sup_{u \in V} \frac{b(u,\mu)}{\| u \|_V} \ge \frac1C \| \mu \|_{\Lambda},
\end{equation}
where $C> 0$ is the constant from Lemma \ref{extension}.
\end{lemma}
\begin{proof}
Let $\mu \in \Lambda$. Recalling that $\Lambda = (H^s(\W^c))'$ and taking into account the extension operator given by Lemma \ref{extension}, we have
\[
\| \mu \|_\Lambda = \sup_{v\in H^s(\W^c)} \frac{b(v, \mu)}{\| v \|_{H^s(\W^c)}} \le
C \sup_{v\in H^s(\W^c)} \frac{b(Ev, \mu)}{\| Ev \|_{V}} \le 
C \sup_{u \in V} \frac{b(u, \mu)}{\|  u \|_{V}} .
\]
\end{proof}

Due to the ellipticity of $a$ on the kernel of $b$ \eqref{eq:ellipticity} and the inf-sup condition \eqref{eq:infsup}, we deduce the well-posedness of the continuous problem by means of the Babu\v{s}ka-Brezzi theory \cite{BoffiBrezziFortin}.
\begin{proposition} \label{prop:well_posedness}
Problem \eqref{eq:cont} admits a unique solution $(u,\lambda) \in V\times\Lambda$, and there exists $C>0$ such that the bound
\[
\| u \|_V + \| \lambda \|_\Lambda \le C \left( \| f \|_{\widetilde{H}^{-s}(\W)} + \| g \|_{H^s(\W^c)} \right)
\]
is satisfied.
\end{proposition}
\begin{remark}
\label{rem:igualesu}
Considering test functions $v\in K$, the first equation of \eqref{eq:cont} implies that $u$ solves  \eqref{eq:cont_direct}, while the second equation of  \eqref{eq:cont} enforces the condition $u\in V_g$. 
\end{remark}


\section{Regularity of solutions} \label{sec:regularity}

 Since the maximum gain of regularity for solutions of the homogeneous problem is ``almost'' half a derivative, from this point on we assume 
 $f \in H^{1/2-s}(\W)$. Moreover, we require the Dirichlet condition $g$ to belong to $H^{s+1/2}(\W^c)$.

As described in \S\ref{ss:direct}, we consider an extension $g^E \in {H^{s+1/2}(\rn)}$ 
and consider the homogeneous problem 
\eqref{eq:homogeneous} with right hand side function equal to $f - (-\Delta)^s g^E$:
\begin{equation*}  
\left\lbrace
  \begin{array}{rll}      
(-\Delta)^s u_0  & = f - (-\Delta)^s g^E  & \text{in } \W, \\
{u_0} & =  0  & \text{in } \W^c.
\end{array}
    \right.
\end{equation*}

Due to Proposition \ref{prop:order}, it follows that $(-\Delta)^s g^E \in H^{{1/2-s}}(\rn)$, with
\[
\| (-\Delta)^s g^E \|_{ H^{{1/2-s}}(\rn)} \le C \| g^E \|_{ H^{{s+1/2}}(\rn)} \le C \| g \|_{ H^{{s+1/2}}(\W^c)},
\]
so that the right hand side function $f - (-\Delta)^s g^E$ belongs to $H^{{1/2-s}}(\W)$. Applying Proposition \ref{prop:regHr} (see also \cite{Grubb, VishikEskin}), we obtain that the solution $u_0\in \widetilde H^{s+1/2-\eps}(\W)$ for $\eps>0$, with
\[
\| u_0 \|_{H^{s+1/2-\eps}(\rn)} \le C{(\eps)} \left( \| f \|_{H^{{1/2-s}}(\W)}  + \| (-\Delta)^s g^E \|_{H^{{1/2-s}}(\W)} \right).
\]
Moreover, as the solution of \eqref{eq:dirichletnh} is given by $u = u_0 + g^E$, we deduce that $u \in H^{s+1/2-\eps}(\rn)$, and
\begin{equation}
\| u \|_{H^{s+1/2-\eps}(\rn)} \le C{(\eps)} \left( \| f \|_{H^{{1/2-s}}(\W)}  + \| g \|_{H^{{s+1/2}}(\W^c)} \right).
\label{eq:reg_u}
\end{equation}
We have proved the regularity of solutions of \eqref{eq:dirichletnh}.

\begin{theorem} \label{teo:regularidadDirect} Let $f \in H^{{1/2-s}}(\W)$ and let $g \in H^{{s+1/2}}(\W^c)$.
Let $u \in H^s(\rn)$ be the solution of \eqref{eq:dirichletnh}. Then, for all $\eps>0,$ $u \in H^{s+1/2-\eps}(\rn)$ and 
there exists $C{=C(\eps)}>0$ such that
\begin{equation*} 
\| u \|_{H^{s+1/2-\eps}(\rn)}  \le \\
  C \left( \| f \|_{H^{{1/2-s}}(\W)}  + \| g \|_{H^{{s+1/2}}(\W^c)} \right). 
\end{equation*}
\end{theorem}

Regularity of the nonlocal normal derivative of the solution is deduced under an additional compatibility hypothesis on the Dirichlet condition. Namely, we assume that $(-\Delta)^s_{\W^c} g \in H^{{1/2-s}}(\W^c)$, where $(-\Delta)^s_{\W^c}$ denotes the regional fractional Laplacian operator \eqref{eq:regional} in $\W^c$. 

\begin{theorem} \label{teo:regularidad} Assume the hypotheses of Theorem \ref{teo:regularidadDirect}, and in addition let $g$ be such that $(-\Delta)^s_{\W^c} g \in H^{{1/2-s}}(\W^c)$. Then, for all $\eps>0,$ $u \in H^{s+1/2-\eps}(\rn)$,  and its nonlocal normal derivative $\lambda \in H^{-s+1/2-\eps}(\W^c)$. Moreover,  there exists $C{=C(\eps)}>0$ such that
\begin{equation*} \begin{split}
\| u \|_{H^{s+1/2-\eps}(\rn)} & + \| \lambda \|_{H^{-s+1/2-\eps}(\W^c)} \le 
C \, \Sigma_{f,g},
\end{split} \end{equation*}
where
\begin{equation} \label{eq:def_sigma}
\Sigma_{f,g} =   \| f \|_{H^{{1/2-s}}(\W)}  + \| g \|_{H^{{s+1/2}}(\W^c)} + \| (-\Delta)_{\W^c}^s g \|_{H^{{1/2-s}}(\W^c)}.
\end{equation}
\end{theorem}
\begin{proof}
We only need to prove that $\lambda \in {H}^{-s+1/2-\eps}(\W^c)$. 
Let $v \in \widetilde{H}^{s-1/2+\eps}(\W^c)$. Since $\lambda = (-\Delta)^s u - (-\Delta)_{\W^c}^s g$ in $\W^c$, we write
\[
\left| \int_{\W^c} \lambda v \right| \le \left( \| (-\Delta)^s u \|_{H^{-s+1/2-\eps}(\W^c)}  + \| (-\Delta)_{\W^c}^s g \|_{H^{-s+1/2-\eps}(\W^c)} \right) \| v \|_{\widetilde H^{s-1/2+\eps}(\W^c)} . 
\]
 
Using Proposition \ref{prop:order}, we deduce
$$
\left| \int_{\W^c} \lambda v \right| \le C  \left( \| u \|_{H^{s+1/2-\eps}(\rn)}  + \| (-\Delta)_{\W^c}^s g \|_{H^{-s+1/2-\eps}(\W^c)} \right) \| v \|_{\widetilde H^{s-1/2+\eps}(\W^c)}
$$
and taking supremum in $v$ we conclude that $\lambda \in H^{-s+1/2-\eps}(\W^c)$,
with
\begin{equation*} \begin{split}
 \| \lambda & \|_{H^{-s+1/2-\eps}(\W^c)} \le 
C \, \Sigma_{f,g},
\end{split} \end{equation*}
where we have used \eqref{eq:reg_u} in the last inequality and the notation \eqref{eq:def_sigma}.
\end{proof}

\begin{remark}
In view of Proposition \ref{prop:order}, it might seem true that for every $\ell \in \R$ and $g \in H^\ell(\W^c)$ it holds that $(-\Delta)_{\W^c}^s g \in H^{\ell-2s}(\W^c)$, which in turn would imply that the hypothesis $(-\Delta)^s_{\W^c} g \in H^{{1/2-s}}(\W^c)$ is superfluous. However, we have not been able neither to prove nor to disprove this claim. As an illustration on what type of additional hypotheses are utilized to ensure this type of behavior of the restricted fractional Laplacian, we refer the reader to \cite[Lemma 5.6]{warma2015}.
\end{remark}

Naturally, the homogeneous case $g\equiv0$ satisfies the assumptions of Theorem \ref{teo:regularidad}. 

\begin{corollary} \label{cor:regularidad}
Let $\W$ be a smooth domain and $f \in H^{{1/2-s}}(\W)$. Let $u \in \widetilde H^s(\W)$ be the solution of \eqref{eq:homogeneous} 
and $\lambda$ be its nonlocal normal derivative. Then, for all $\eps>0,$ it holds that $\lambda \in H^{-s+1/2-\eps}(\W^c)$ and
\begin{equation*}
\| \lambda \|_{H^{-s+1/2-\eps}(\W^c)} \le C(n,s,\W,\eps)  \| f \|_{H^{{1/2-s}}(\W)} . 
\end{equation*}
\end{corollary}

\begin{remark} \label{rem:singular} We illustrate the sharpness of the regularity estimate for the nonlocal derivative from Theorem \ref{teo:regularidad} (or from Corollary \ref{cor:regularidad}) with the following simple example. Let $\W = (-1,1)$ and consider the problem 
\[
\left \lbrace
\begin{array}{rl}
(-\Delta)^s u = 1 & \text{ in } (-1,1), \\
u = 0 & \text{ in } \R\setminus(-1,1),
\end{array}
\right.
\]
whose solution is given by $u(x) = c(s) (1 - x^2)_+^s$ for some constant $c(s) > 0$ (see, for example, \cite{Getoor}). 
We focus on the behavior of $\mathcal{N}_s u$ near the boundary of $\W$; for instance, let $x \in (1,2)$. Basic manipulations allow to derive the bound
\[ 
\left| \mathcal{N}_s u (x) \right| > \frac{C(s)}{(x-1)^s}.
\]
Next, given $\alpha \in (0,1)$, observe that $(x-1)^\alpha \in H^\ell(1,2)$ if and only if $\ell < \alpha + 1/2.$ Thus, by duality, we conclude that $\mathcal N_su \notin H^{{1/2-s}}(1,2)$.
 
The reduced regularity of the nonlocal normal derivative near the boundary does not happen as an exception but is what should be expected in general. Indeed, following \cite{ABBM}, let $f:(-1,1)\to\R$ be a function such that its coefficients $f_j$ (in the expansion with respect to the basis of the so-called Gegenbauer polynomials $\left\{ C^{(s+1/2)}_j \right\}$) satisfy either
\[\sum_{j=0}^\infty \frac{f_j \, j !}{\Gamma(2s+j+1)} \, C_j^{(s+1/2)}(-1) \neq 0 \, 
\text{ or } \sum_{j=0}^\infty \frac{f_j  \,  j !}{\Gamma(2s+j+1)} \, C_j^{(s+1/2)}(1) \neq 0. 
\]
Then, the solution to \eqref{eq:homogeneous} is given by $u(x) = (1-x^2)_+^s \phi(x)$, where $\phi$ is a smooth function that does not vanish as $|x| \to 1$ (cf. \cite[Theorem 3.14]{ABBM}). Therefore, the same argument as above applies: the nonlocal derivative of the solution of the homogeneous Dirichlet problem belongs to $H^{-s+1/2-\eps}(\R \setminus (-1,1))$, and the $\eps > 0$ cannot be removed.

We remark that in the limit $s\to 1$ the nonlocal normal derivatives concentrate mass towards the boundary of the domain, so that \cite{dipierro2014nonlocal}
\[
\lim_{s\to 1} \int_{\W^c} \mathcal{N}_s u \, v = \int_{\pp\W} \frac{\pp u}{\pp n} \, v \quad \forall u,v \in C^2_0(\rn).
\]
This estimate also illustrates the singular behavior of $\mathcal{N}_s u$ near the boundary of $\W$.
\end{remark}

\section{Finite Element approximations} \label{sec:fe_approximations}
In this section we begin the study of finite element approximations to problem \eqref{eq:cont}. Here we assume the Dirichlet datum $g$ to have bounded support.
This assumption allows to simplify the error analysis of the numerical method we propose in this work, but it is not necessary.  In the next section, estimates for data not satisfying such hypothesis are deduced.

\subsection{Finite element spaces}
Given ${H>1}$ big enough, we denote by $\W_H$ a domain containing $\W$ and such that
\begin{equation}
c H \le \min_{x\in \partial \W, \, y \in \partial \W_H} d(x,y) \le \max_{x\in \partial \W, \, y \in \partial \W_H} d(x,y) \le C H, 
\label{eq:def_OmegaH}
\end{equation}
where $c, C$ are constants independent of $H$. 
We set conforming simplicial meshes on $\W$ and $\W_H \setminus \W$, in such a way that the resulting partition of $\W_H$ remains admissible.  Moreover, to simplify our analysis, we assume the family of meshes to be globally quasi-uniform.

\begin{remark}
The parameter $H$ depends on the mesh size $h$ in such a way that as $h$ goes to $0$, $H$ tends to infinity.  
The purpose of $\W_H$ is twofold: in first place, to provide a domain in which to implement the finite element approximations. 
In second place, the behavior of solutions may be controlled in the complement of $\W_H$. 
Assuming $g$ to have bounded support implies that, for $h$ small enough, the domain $\W_H$ contains the support of the Dirichlet datum $g$.
Moreover, since there is no reason to expect $\lambda$ to be compactly supported, taking $H$ depending adequately on $h$ ensures that 
the decay of the nonlocal derivative in $\W^c_H$ is of the same order as the approximation error of $u$ and $\lambda$ within $\W_H$. 
\end{remark}

We consider nodal basis functions
\[
\phii_1, \ldots, \phii_{N_{int}}, \phii_{N_{int}+1}, \ldots , \phii_{N_{int}+N_{ext}} ,
\]
where the first $N_{int}$ nodes belong to the interior of $\W$ and the last $N_{ext}$ to ${\W_H \setminus \W}$.
The discrete spaces we consider consist of continuous, piecewise linear functions:
\begin{align*}
& V_h = \text{span } \{\phii_1, \ldots, \phii_{N_{int}+N_{ext}}  \}, \\
& K_h = \text{span } \{\phii_1, \ldots, \phii_{N_{int}}  \}, \\
& \Lambda_h = \text{span } \{ \phii_{N_{int}+1}, \ldots, \phii_{N_{int}+N_{ext}}  \}. 
\end{align*}
The spaces $V_h$ and $\Lambda_h$ are endowed with the $\| \cdot \|_V$ and $\| \cdot \|_\Lambda$ norms, respectively.
We set the discrete functions to vanish on $\partial \W_H$, so that $V_h \subset \widetilde{H}^{3/2-\eps}(\W_H)$.

\subsection{The mixed formulation with a Lagrangian multiplier}
The discrete problem reads: find $(u_h, \lambda_h) \in V_h \times \Lambda_h$ such that
\begin{equation}\label{eq:discrete}
\begin{split}
a(u_h, v_h) - b(v_h, \lambda_h) = F(v_h) \ & \forall v_h \in V_h, \\
b(u_h, \mu_h) = G(\mu_h) \ & \forall \mu_h \in \Lambda_h .
\end{split}
\end{equation}
Notice that the space $K_h$ coincides with the kernel of the restriction of $b$ to $\Lambda_h$
and consists of piecewise linear functions  over the triangulation of $\W$ that vanish on $\partial \W$

To verify the well-posedness of the discrete problem \eqref{eq:discrete}, we need to show that the bilinear form $a$ is coercive on $K_h$ and that the discrete inf-sup condition for the bilinear form $b$ holds.

\begin{lemma} There exists a constant $C > 0$, independent of $h$ {and $H$}, such that for all $v_h \in K_h$,
\begin{equation} \label{eq:disc_coercivity}
a(v_h, v_h) \ge C \| v_h \|_{V}^2 . 
\end{equation}
\end{lemma}
\begin{proof}
Observe that $K_h$ is a subspace of the continuous kernel $K$ given by \eqref{eq:kernel}.
The lemma follows by the coercivity of $a$ on $K$.
\end{proof}

In order to prove the discrete inf-sup condition, we utilize a projection over the discrete space. Since $V_h \subset \widetilde H^{3/2-\eps}(\W_H)$ for all $\eps > 0$, it is possible to define the $L^2$-projection of functions in the dual space of ${\widetilde H^{3/2-\eps}}(\W_H).$
Namely, we consider $P_h : H^{-\sigma}(\W_H) \to V_h$ for $0\le \sigma\le1$, the operator characterized by
\[
\int_{\W_H} (w - P_h w) \, v_h = 0 \quad \forall v_h \in V_h . 
\]

The following property will be useful in the sequel.

\begin{lemma} \label{projection} 
Let $0<\sigma<1$, and assume the family of meshes to be quasi-uniform. Then, there exists a constant $C$, independent of $h$ {and $H$}, such that 
\[
\| P_h w \|_{ H^\sigma(\W_H)} \le C \| w \|_{H^\sigma(\W_H)}
\]
for all $w \in  H^\sigma(\W_H)$.
\end{lemma}
\begin{proof}
The proof follows by interpolation. On the one hand, the $L^2$-stability estimate 
\[
\| P_h w \|_{L^2(\W_H)} \le \| w \|_{L^2(\W_H)}
\]
is obvious.
On the other hand, the $H^1$ bound
\begin{equation}
\label{eq:estH1}
\| P_h w \|_{H^1(\W_H)} \le C \| w \|_{H^1(\W_H)} 
\end{equation}
is a consequence of a global inverse inequality (see, for example \cite{Auricchio}). 
\new{Because $P_h$ commutes with dilations, a scaling argument allows to show that $C$ can be taken independent of $H$. Indeed, we can assume --after a translation, if needed-- that $\W_H$ is a ball $B_R$ of radius $R\ge 1$ centered at the origin.
Denote by  $\hat{P}_h$ the $L^2$-projections over meshes in $B_1$. Then, for every $\hat w\in H^1(B_1)$ and every quasi-uniform mesh it holds that
\[
\| \hat{P}_h \hat w \|_{H^1(B_1)} \le C_1 \| \hat w \|_{H^1(B_1)} ,
\]
where $C_1$ is a \emph{fixed} constant. 
Next, define $T \colon B_1\to B_R$ by $T(\hat x)=R\hat x$, and, for each $ w\in H^1(B_R)$, the function $ w \circ T=\hat w\in H^1(B_1)$. 
Every quasi-uniform mesh $\T$ on $B_R$ with mesh size $h$ is in correspondence with a quasi-uniform mesh on $B_1$ with mesh size $\frac{h}{R}$ through the obvious identification $\T = T (\hat \T)$. For these meshes we have the identity $\widehat{P_h w}=\hat{P}_{\frac{h}{R}} \hat{w}$ and hence, changing variables,
\[
\| \nabla P_h w\|_{L^2(B_R)}= R^{\frac{n}{2}-1}
\| \nabla \hat{P}_{\frac{h}{R}} \hat w\|_{L^2(B_1)}\le 
C_1 R^{\frac{n}{2}-1} \left( \| \nabla \hat w\|_{L^2(B_1)} + \|\hat w\|_{L^2(B_1)} \right).
\]
Therefore,
\[
\| \nabla P_h w\|_{L^2(B_R)}\le C_1
\left( \| \nabla w\|_{L^2(B_R)} +\frac{1}{R} \| w\|_{L^2(B_R)} \right),
\]
and then
\[
\| \nabla P_h w\|_{L^2(B_R)}\le 2C_1 \|  w \|_{H^1(B_R)}.
\]
Since bounds for $ \|  P_h w\|_{L^2(B_R)}$ are immediate, \eqref{eq:estH1} follows.}
\end{proof}

\begin{remark}
The global quasi-uniformity hypothesis could actually be weakened and substituted by the ones from \cite{BramblePasciakSteinbach, Carstensen, CrouzeixThomee}. In these works, meshes are required to be just locally quasi-uniform, but some extra control on the change in measures of neighboring elements is needed as well.
\end{remark}

Stability estimates in negative-order norms are obtained by duality.
\begin{lemma}
Let $0 \le \sigma \le 1$, and assume the family of meshes to be quasi-uniform. Then, there is a constant $C$, independent of $h$ and $H$, such that 
\begin{equation*} 
\| P_h w \|_{\widetilde H^{-\sigma}(\W_H)} \le C \| w \|_{\widetilde H^{-\sigma}(\W_H)}
\end{equation*}
for all $w \in \widetilde H^{-\sigma}(\W_H)$.
\end{lemma}
\begin{proof}
Consider $v \in H^\sigma(\W_H)$. We have 
\[
\int_{\W_H} P_hw \, v \, = \new{\int_{\W_H} P_h w \, P_h v} \,=\int_{\W_H} w \, P_hv \, \le \| w \|_{\widetilde  H^{-\sigma}(\W_H)} \| P_h v \|_{ H^{\sigma}(\W_H)} .
\]
The proof follows by the $ H^\sigma$-stability of $P_h$.
\end{proof}

\begin{remark}
For simplicity, the previous lemma was stated for functions defined in $\W_H$, but clearly it is also valid over $\W_H \setminus \W$:
\begin{equation} \label{eq:estab_neg}
\| P_h w \|_{\widetilde H^{-\sigma}(\W_H\setminus\W)} \le C \| w \|_{\widetilde H^{-\sigma}(\W_H\setminus\W)} \quad \forall w \in \widetilde H^{-\sigma}(\W_H\setminus\W).
\end{equation}
 
\new{For the sake of completeness, since the scaling argument does not carry over straightforwardly, we sketch a proof of the stability estimate 
$$\| P_h w \|_{ H^\sigma(\W_H\setminus\W)} \le C \| w \|_{H^\sigma(\W_H\setminus\W)}.$$ As in the proof of Lemma \ref{projection}, it suffices to show
\begin{equation}
\label{eq:estH1diff}
\| P_h w \|_{H^1(\W_H\setminus \W)} \le C \| w \|_{H^1(\W_H\setminus \W)} 
\end{equation}
with a fixed $C$, and then conclude by interpolation with the $L^2$ estimate.}

\new{ Consider a smooth truncation function $0\le \psi \le 1$, such that $\psi=1$ in $\W_1:=\{x\in \R^n \colon d(x,\W)<1\}$. 
Assume the support of $\psi$ is contained in a fixed open ball $B_r$ with radius $r$.
Thus,  for $H$ large enough (namely, for $h$ small enough), $B_r\subset \W_H$. 
Given $w\in H^1(\W_H\setminus \W)$, we write $w=w\psi + w(1-\psi)$ and therefore
we just need to bound 
\[
\| P_h (w \psi) \|_{H^1(\W_H\setminus \W)} \ \mbox{ and } \ \| P_h [w(1- \psi)] \|_{H^1(\W_H\setminus \W)}.
\] }

\new{ Because $r$ is fixed, if $h$ is small enough the former norm coincides with the norm over $B_r\setminus \W$, since
$B_r$ is open and contains the support of $\psi$. Moreover, since $\psi$ is smooth, we bound  
\[ 
\begin{aligned}
\| P_h (w \psi) \|_{H^1(\W_H\setminus \W)} & = \| P_h (w \psi) \|_{H^1(B_r\setminus \W)} \le C(r,\W) \| w \psi \|_{H^1(B_r\setminus \W)} \\
& \le C(r, \W, \psi) \| w \|_{H^1(\W_H\setminus \W)}.
\end{aligned} 
\]
On the other hand, considering a zero-extension within $\W$ and using \eqref{eq:estH1} and the smoothness of $\psi$ we deduce
\[ 
\begin{aligned}
\| P_h [w(1- \psi)] \|_{H^1(\W_H\setminus \W)} & = \| P_h [w(1- \psi)] \|_{H^1(\W_H)} \le C\| w(1- \psi) \|_{H^1(\W_H)} \\  
& \le C \| w(1- \psi) \|_{H^1(\W_H\setminus \W)} \le C\| w \|_{H^1(\W_H\setminus \W)},
\end{aligned}
\]
 with a final constant $C$ depending only on $r$ and $\psi$. 
From these estimates, \eqref{eq:estH1diff} follows immediately, and in consequence, we obtain the bound \eqref{eq:estab_neg}. }
\end{remark}

\begin{proposition}
Let $s \ne \frac12$. Then, there exists a constant $C$, independent of $h$ {and $H$}, such that the following discrete inf-sup condition holds:
\begin{equation}
\sup_{v_h \in V_h} \frac{b(v_h, \mu_h)}{\| v_h \|_{V}} \ge
C \| \mu_h \|_\Lambda \quad \forall \mu_h \in \Lambda_h.
\label{eq:infsup_disc}
\end{equation}
\end{proposition}
\begin{proof}
In first place, let $E$ be the extension operator given by Lemma \ref{extension}
{(replacing $\W$ with $\W^c$ there)} and $P_h$ the $L^2$-projection considered 
in this section. For simplicity of notation, we write, for $v \in H^s(\W^c)$, $P_h(Ev) = P_h \left((Ev) \big|_{\W_H} \right).$ Taking into account 
the fact that $P_h(Ev) \in \widetilde{H}^s(\W_H)$ and the continuity of these operators, it is clear that
\[
\| P_h (Ev) \|_V = \| P_h (Ev) \|_{\widetilde H^s(\W_H)} \le C \| v \|_{H^s(\W^c)} \quad \forall v \in H^s(\W^c), 
\]
which in turn allows us to use $P_h(Ev)$ as a Fortin operator.

Indeed, let $\mu_h \in \Lambda_h$, $v \in H^s(\W^c)$ and write 
\[ \begin{aligned}
\sup_{v_h \in V_h} \frac{b(v_h, \mu_h)}{\| v_h \|_{V}}  & \ge
\frac{b(P_h(Ev), \mu_h)}{ \| P_h(Ev) \|_V} \ge C \frac{b(v, \mu_h)}{ \| v \|_{H^s(\W^c)}}.
\end{aligned} \]
Using the fact that $v$ is arbitrary together with \eqref{eq:infsup}, we deduce \eqref{eq:infsup_disc}.
\end{proof}

\begin{remark} \label{rem:s_not_1/2}
{The previous proposition is the basis for the stability of the mixed numerical method we propose in this paper. The proof works only for $s\neq \frac12$ and thus from this point on we asume that to be the case. 
However, we remark that the experimental orders of convergence we have obtained for $s = \frac12$ agree with those expected by the theory by taking the limit $s\to \frac12$, supporting the fact that this drawback is a mere limitation of our proof.}
\end{remark}

Due to the standard theory of finite element approximations of saddle point problems \cite{BoffiBrezziFortin}, we deduce the following estimate.
\begin{proposition} \label{prop:cea}
Let $(u,\lambda) \in V\times \Lambda$ and $(u_h,\lambda_h) \in V_h \times \Lambda_h$ be the respective solutions of problems \eqref{eq:cont} and \eqref{eq:discrete}. Then there exists a constant $C$, independent of $h$ {and $H$}, such that
\begin{equation} \label{eq:cea}
\| u - u_h \|_V + \| \lambda - \lambda_h \|_\Lambda \le 
C \left( \inf_{v_h \in V_h} \| u - v_h \|_V +  \inf_{\mu_h \in \Lambda_h} \| \lambda - \mu_h \|_\Lambda \right) . 
\end{equation}
\end{proposition}

In order to obtain convergence order estimates for the finite element approximations under consideration, it remains to estimate the infima on the right hand side of \eqref{eq:cea}.
Within $\W_H$, this is achieved by means of a quasi-interpolation operator \cite{Clement, ScottZhang}.
We denote such an operator by $\Pi_h$; depending on whether discrete functions are required to have zero trace or not, $\Pi_h$ could be either the Cl\'ement or the Scott-Zhang operator. For these operators, it holds that (see, for example, \cite{Ciarlet}) 
\begin{equation}
 \label{eq:sz_estimate}
\| v -\Pi_h v \|_{H^{t}(\W)} \le C h^{\sigma - t} \| v  \|_{H^{\sigma}(\W)}
\quad \forall v \in H^{\sigma}(\W), \ 0 \le t \le \sigma\le 2.
\end{equation}

{Since this estimate is applied later to $\W_H\setminus\W$ it is important to stress that the 
constant can be taken independent of the diameter of $\W$. This is indeed the case due to the fact
that \eqref{eq:sz_estimate} is obtained   by summing \emph{local} estimates on stars (see e.g., \cite{AB, Ciarlet}). }

\begin{lemma} Given $v \in L^2(\W_H \setminus \W)$ and $0\le \sigma \le 1$, the following estimate holds:
\begin{equation}
\| v - P_h v \|_{\widetilde H^{-\sigma}(\W_H \setminus \W)} \le C h^{\sigma} \| v \|_{L^2(\W_H \setminus \W)}.
\label{eq:aprox_neg}
\end{equation}
{The constant $C$ is independent of $h$ and $H$.}
\end{lemma}
\begin{proof}
Let $v \in L^2({\W_H}\setminus \W)$. Given $\phii \in H^\sigma(\W_H\setminus\W)$, considering the quasi-interpolation operator $\Pi_h$ and taking into account that 
$(v-P_h v) \perp V_h$,
\begin{align*}
\frac{\int_{\W_H\setminus \W} (v - P_h v) \phii}{ \| \phii \|_{ H^\sigma(\W_H\setminus\W)}} & = 
\frac{\int_{\W_H\setminus \W} (v - P_h v) (\phii - \Pi_h \phii)}{ \| \phii \|_{H^\sigma(\W_H\setminus\W)}} \le \\
& \le \|v - P_h v \|_{L^2(\W_H \setminus \W)} \frac{\| \phii - \Pi_h \phii \|_{L^2(\W_H \setminus \W)}}{ \| \phii \|_{H^\sigma(\W_H\setminus\W)}} .
\end{align*}
Combining well-known approximation properties of $\Pi_h$ with the trivial estimate $\|v - P_h v \|_{L^2(\W_H \setminus \W)} \le \|v \|_{L^2(\W_H \setminus \W)}$, we conclude the proof.
\end{proof} 

For the following we need to define restrictions in negative order spaces. 
Let $\sigma \in (0,1)$ and choose a fixed cutoff 
function $\eta \in C^\infty(\W^c)$ such that
\begin{equation}
 \label{eq:cutoff}
0\le \eta\le 1, \quad  \text{supp}(\eta)\subset \overline{\W}_H \setminus \W,  \quad \eta(x)=1 \quad \mbox{in} \quad \W_{H-1} \setminus \W   .
\end{equation}

Define the operator {$T_\eta:\; H^\sigma(\W_H\setminus \W) \to H^\sigma(\W^c)$} that multiplies by $\eta$ 
{any extension to $\W^c$ of functions in $H^\sigma(\W_H\setminus \W)$},
that is, $T_\eta(\psi):=\eta \psi$. We have {$\|T_\eta(\psi)\|_{H^\sigma(\W^c)}\le C \|\psi\|_{H^\sigma(\W_H\setminus \W)}$}, with a constant that does not depend on $H$  (use interpolation from the obvious cases $\sigma=0$ and $\sigma=1$).

Then, $T_\eta$ can be extended to negative-order spaces, $T_\eta :\widetilde H^{-\sigma}(\W^c)\to \widetilde H^{-\sigma}(\W_H \setminus \W)$.
Consider an element $\mu\in \widetilde H^{-\sigma}(\W^c)$, and define $T_\eta$ by means of 
\[
\langle T_\eta(\mu), \psi \rangle=\langle \mu , \eta \psi \rangle. 
\]
The continuity $\|T_\eta(\mu)\|_{\widetilde H^{-\sigma}(\W_H\setminus \W)}\le C \|\mu\|_{\widetilde H^{-\sigma}(\W^c)},$
follows easily from the continuity in positive spaces. Notice that similar considerations hold for $T_{1-\eta} : \widetilde H^{-\sigma}(\W^c)\to \widetilde H^{-\sigma}(\W_{H-1}^c)$. A localization estimate for negative-order norms using these maps reads as follows.

\begin{lemma} \label{lemma:triangular}
The following identity holds for all $\mu \in \widetilde H^{-\sigma}(\W^c)$:
\[
\| \mu \|_{\widetilde H^{-\sigma}(\W^c)} \le \| T_\eta (\mu) \|_{\widetilde H^{-\sigma}(\W_H\setminus\W)} + \| T_{1-\eta} (\mu) \|_{\widetilde H^{-\sigma}(\W_{H-1}^c)}.
\]
\end{lemma}
\begin{proof}
We first notice that, for every $\psi \in H^{\sigma}(\W^c)$, it holds that
\[
\psi = T_\eta \left( \psi \big|_{\W_H\setminus\W} \right) + T_{1-\eta} \left( \psi \big|_{\W_{H-1}^c} \right),
\]
and that 
\[ \begin{aligned}
\| T_\eta \left( \psi \big|_{\W_H\setminus\W} \right)\|_{H^{\sigma}(\W^c)} & \le \|  \psi \big|_{\W_H\setminus\W} \|_{H^{\sigma}(\W_H\setminus\W)}, \\
\| T_{1-\eta} \left( \psi \big|_{\W_{H-1}^c} \right)\|_{H^{\sigma}(\W^c)} & \le \|  \psi \big|_{\W_{H-1}^c} \|_{H^{\sigma}(\W_{H-1}^c)} . 
\end{aligned} \]
So, given $\mu \in \widetilde H^{-\sigma}(\W^c)$, it follows that
\begin{equation} \label{eq:linearity_mu}
\frac{\langle \mu, \psi \rangle}{\| \psi \|_{H^{\sigma}(\W^c)}} \le 
\frac{\langle T_\eta (\mu), \psi \big|_{\W_H\setminus\W} \rangle}{\| \psi\big|_{\W_H\setminus\W} \|_{H^{\sigma}(\W_H\setminus\W)}}  
+ \frac{\langle T_{1-\eta} (\mu), \psi\big|_{\W_{H-1}^c}  \rangle}{\| \psi\big|_{\W_{H-1}^c}  \|_{H^{\sigma}(\W_{H-1}^c)}} 
\end{equation}
for all $\psi \in H^{\sigma}(\W^c)$. The proof follows by taking suprema in both sides of the inequality above.
\end{proof}

\begin{remark} \label{remark:triangular}
 From \eqref{eq:linearity_mu}, it is apparent that, if $\mu \in \widetilde H^{-\sigma}(\W^c)$ and  $\nu \in \widetilde H^{-\sigma}(\W_H \setminus \W)$, then 
\[
\| \mu - \nu \|_{\widetilde H^{-\sigma}(\W^c)} \le \| T_\eta (\mu) - \nu  \|_{\widetilde H^{-\sigma}(\W_H\setminus\W)} + \| T_{1-\eta} (\mu) \|_{\widetilde H^{-\sigma}(\W_{H-1}^c)}.
\]
\end{remark}

In order to simplify notation, in the sequel we just write $\eta\mu $ and $(1-\eta)\mu$ for $T_\eta(\mu)$ and $T_{1-\eta}(\mu)$, respectively.

Next, we estimate the approximation errors within the meshed domain.

\begin{proposition} \label{prop:interpolation}
The following estimates hold:
\begin{align}
\inf_{v_h \in V_h} \| u - v_h \|_{H^s(\W_H)}  & \le 
C \,{h^{1/2-\eps}} \Sigma_{f,g},
\label{eq:interpolation_H_u} \\
  \inf_{\mu_h \in \Lambda_h} \| \eta \lambda - \mu_h \|_{\widetilde H^{-s}(\W_H\setminus\W)} & \le
C \,{h^{1/2-\eps}} \Sigma_{f,g}, \label{eq:interpolation_H_nsu}
\end{align}
where $\Sigma_{f,g}$ is given by \eqref{eq:def_sigma} and $\eta$ is the cutoff function in \eqref{eq:cutoff}.
\end{proposition}
\begin{proof}
Estimate \eqref{eq:interpolation_H_u} is easily attained by taking into account that $u$ vanishes on $\W_H^c$ (because we are assuming that the support of $g$ is bounded), and applying the regularity estimate \eqref{eq:reg_u} jointly with approximation identities for quasi-interpolation operators.

In order to prove \eqref{eq:interpolation_H_nsu},  in first place we assume $s< 1/2$, 
so that $\eta \lambda \in  L^2(\W_H \setminus \W)$ {by Theorem~\ref{teo:regularidad}}. Set $\mu_h = P_h (\eta \lambda)$, then  applying \eqref{eq:aprox_neg}, approximation properties of $P_h$ and the continuity of $T_\eta:H^{-s+1/2-\eps}(\W_H\setminus \W) \to  H^{-s+1/2-\eps}(\W^c)$, we obtain \eqref{eq:interpolation_H_nsu} immediately.

Meanwhile, if $s > 1/2$, considering $\sigma = s$ in \eqref{eq:estab_neg} and \eqref{eq:aprox_neg}, we obtain:
\begin{align*}
\| w - P_h w \|_{\widetilde H^{-s}(\W_H \setminus \W)} &  \le C \| w \|_{\widetilde H^{-s}(\W_H \setminus \W)} \\
\| w - P_h w \|_{\widetilde H^{-s}(\W_H \setminus \W)} & \le C h^{s} \| w \|_{L^2(\W_H \setminus \W)} .
\end{align*}
Interpolating these two identities, recalling the regularity of $\lambda$ given by Theorem \ref{teo:regularidad} 
and the continuity of $T_\eta$, we deduce that 
\begin{align*}
\| \eta \lambda - P_h (\eta \lambda) \|_{\widetilde H^{-s}(\W_H \setminus \W)} &\le C h^{1/2 - \eps}\| \lambda \|_{H^{-s+1/2-\eps}({\W^c})} \le C  h^{1/2 - \eps} \,  \Sigma_{f,g}.
 \end{align*}
\end{proof}

As the norms in both $V$ and $\Lambda$ involve integration on unbounded domains and the discrete functions vanish outside $\W_H$, in order to estimate the infima in \eqref{eq:cea}, we need to rely on identities that do not depend on the discrete approximation but on the behavior of $u$ and $\lambda$. For the term corresponding to the norm of $u$, Corollary \ref{cor:norma} suffices (as long as $\text{supp}(g) \subset \W_H$), whereas for the nonlocal derivative contribution it is necessary to formulate an explicit decay estimate.
\begin{proposition} \label{prop:dec_nsu} Let $\W_H$ be such that $\text{supp}(g) \subset \W_H$. 
Then, there exists a constant $C$, independent of $f$, $g$ {and $H$}, such that the estimate
\[ \begin{aligned}
\| (1-\eta) \lambda \|_{\widetilde H^{-s}(\W_{H-1}^c)} & \le  \| \lambda \|_{L^2(\W_{H-1}^c)} \\\
& \le C H^{-(n/2 + 2s)}  \left( \| f \|_{H^{\regf}(\W)}  + \| g \|_{H^{\regg}(\W^c)} \right)
\end{aligned} \]
holds, where $\eta$ is the cutoff function from \eqref{eq:cutoff}.
\end{proposition}
\begin{proof}
It is evident that 
\[ 
\| (1-\eta)\lambda \|_{\widetilde H^{-s}(\W_{H-1}^c)} \le \| (1-\eta)\lambda \|_{L^2(\W_{H-1}^c)} \le \| \lambda \|_{L^2(\W_{H-1}^c)}.
\]
Given $x \in \W_{H-1}^c$, it holds that
\[
|\lambda(x)| \le C(n,s) \left[ \int_\W \frac{|u(y)|}{|x-y|^{n+2s}} \, dy 
+  |g(x)| \int_\W \frac{1}{|x-y|^{n+2s}} \, dy \right],
\]
and therefore
\begin{equation} \label{eq:estimacion_lambda} \begin{aligned}
\| \lambda \|_{L^2(\W_{H-1}^c)}^2 \le C \bigg[ & \int_{\W_{H-1}^c} \left(\int_\W \frac{|u(y)|}{|x-y|^{n+2s}} \, dy \right)^2  dx
 \\
& + \int_{\W_{H-1}^c}  |g(x)|^2 \left( \int_\W \frac{1}{|x-y|^{n+2s}} \, dy \right)^2 dx \bigg] .
\end{aligned} 
\end{equation}
We estimate the two integrals in the right hand side above separately.
As for the first one, consider the auxiliary function $\w: \W \to \mathbb{R}$, 
\[
\w(y) = \left(  \int_{\W_{H-1}^c} \frac{1}{|x-y|^{2(n+2s)}} \right)^{1/2};
\]
integrating in polar coordinates and noticing that $(H-1)^{-(n/2+2s)} \le C H^{-(n/2+2s)}$, we deduce
\[
| \w (y) | \le C H^{-(n/2+2s)} \quad \forall y \in \W ,
\]
and so, $ \| \w \|_{L^2(\W)} \le C H^{-(n/2+2s)}$.
As a consequence, 
applying Minkowski's integral inequality, the Cauchy-Schwarz inequality and the previous estimate for $ \| \w \|_{L^2(\W)}$, we obtain
\[ \begin{aligned}
\int_{\W_{H-1}^c} \left(\int_\W \frac{|u(y)|}{|x-y|^{n+2s}} \, dy \right)^2  dx & \le C \left( \int_\W |u(y)| \, |\w(y)| \, dy \right)^2  \\
& \le  C H^{-2(n/2+2s)} \| u \|_{L^2(\W)}^2. 
\end{aligned} \]
Finally, the $L^2$-norm of $u$ is controlled in terms of the data (see, for example, \eqref{eq:reg_u}).

As for the second term in the right hand side in \eqref{eq:estimacion_lambda}, it suffices to notice that for $x \in \W_{H-1}^c$, it holds
\[
\int_\W \frac{1}{|x-y|^{n+2s}} \, dy \le C H^{-(n+2s)}.
\]
This implies that
\[
\int_{\W_{H-1}^c}  |g(x)|^2 \left( \int_\W \frac{1}{|x-y|^{n+2s}} \, dy \right)^2 dx \le
C H^{-2(n+2s)} \| g \|_{L^2(\W_{H-1}^c)}^2 ,
\]
and concludes the proof.
\end{proof}

\begin{remark} \label{rem:orden_H}
As the finite element approximation $u_h$ to $u$ in $\W_H$ has an $H^s$-error of order $h^{1/2-\eps},$ we need the previous estimate for the nonlocal derivative to be at least of the same order. Thus, we require $H^{-(n/2+2s)} \le C h^{1/2}$, that is, $H\ge C h^{-1/(n+4s)}.$ 
\end{remark}

Collecting the estimates we have developed so far, we are ready to prove the following.

\begin{theorem} \label{teo:convergencia_bounded}
Let $\W$ be a bounded, smooth domain, $f\in H^{{1/2-s}}(\W)$ and $g \in H^{{s+1/2}}(\W^c)$. Moreover, assume that $g$ has bounded support and consider $\W_H$ according to \eqref{eq:def_OmegaH}, with $H \gtrsim h^{-1/(n+4s)}.$ For the finite element approximations considered in this work and $h$ small enough, the following a priori estimates hold:
\begin{align}
& \| u - u_h \|_{V}  \le C h^{1/2 - \eps} \Sigma_{f,g},  \label{eq:aprox_u} \\ 
& \| \lambda - \lambda_h \|_{\Lambda} \le C h^{1/2-\eps} \Sigma_{f,g}. \label{eq:aprox_nsu}
\end{align}
for a constant $C$ depending on $\eps$ but independent of $h$, $H$, $f$ and $g$, and $\Sigma_{f,g}$ defined by \eqref{eq:def_sigma}.
\end{theorem}
\begin{proof}
In order to obtain the above two inequalities, it is enough to estimate the infima in \eqref{eq:cea}. 
Since $g$ is boundedly supported and $H \to \infty$ as $h \to 0$, if $h$ is small enough then $\mbox{supp}(g) \subset \W_H$. So, $ u - v_h \in \widetilde{H}^{s}(\W_H)$ for all $v_h \in V_h$ and thus we may apply Corollary \ref{cor:norma} (or Remark \ref{remark:un_medio} if $s=1/2$)
together with \eqref{eq:interpolation_H_u}:
\begin{align*}
\inf_{v_h \in V_h} \| u - v_h \|_V & \le C \inf_{v_h \in V_h} \| u - v_h \|_{H^s(\W_H)} \le 
 C h^{1/2 - \eps } \, \Sigma_{f,g}.
\end{align*}

The infimum involving the nonlocal derivative is estimated as follows. Consider the cutoff function $\eta$ from \eqref{eq:cutoff}. Since $\mu_h$ vanishes in $\W_H^c$, using Remark \ref{remark:triangular}, we have 
\[
\inf_{\mu_h \in \Lambda_h} \| \lambda - \mu_h \|_{\Lambda} \le \inf_{\mu_h \in \Lambda_h} \|\eta \lambda - \mu_h \|_{\widetilde H^{-s}(\W_H \setminus \W)} + \| (1-\eta)\lambda \|_{\widetilde H^{-s}(\W_{H-1}^c)} .
\]
The first term on the right hand side is bounded by means of estimate \eqref{eq:interpolation_H_nsu}, whereas for the second one we apply Proposition \ref{prop:dec_nsu} and notice that the choice of $H$ implies that $H^{-(n/2+2s)} \le C h^{1/2}.$
It follows that 
\[
\inf_{\mu_h \in \Lambda_h} \| \lambda - \mu_h \|_{\Lambda}  \le C h^{1/2 - \eps} \Sigma_{f,g},
\]
and the proof is completed.
\end{proof}

\subsection{The Direct Method}
As it is already mentioned in the introduction, in this work we mainly focus on the
mixed formulation. Nonetheless, here we provide some details regarding the 
direct discrete formulation. 
We consider the discrete problem: find $u_h\in V_{h,g_h}$ such that
\begin{equation}
 \label{eq:discreteDirect}
 a(u_h, v_h)  = F(v_h) \quad \forall v_h \in K_h, 
\end{equation}
where $V_{h,g_h}$ is the subset of $V_h$ of functions that agree with $g_h$
in $\W_H\setminus\W$. 

The function $g_h$ is chosen as an approximation of $g$; for instance, we may consider $g_h=\Pi_h(g)$. As a consequence,  it holds that $\|g-g_h\|_{H^s(\W^c)}\le Ch^{1/2-\eps} \|g\|_{H^{\regg}(\W^c)}$. Let $u$ and $u^{(h)}$ be the solutions of the continuous problem with right hand side $f$ and Dirichlet conditions  $g$ and $g_h$, respectively. Using Proposition \ref{prop:well_posedness_direct}, we deduce that 
\[
\|u- u^{(h)}\|_V\le C h^{1/2-\eps}\|g\|_{H^{\regg}(\W^c)}.
\]
Therefore, in order to bound $\|u- u_h\|_V$ it is enough to bound 
$\| u^{(h)}-u_h\|_V$. However, if $\text{supp}(g)\subset \W_H$, then $u^{(h)}-u_h\in K=\widetilde H^s(\W)$ and 
due to the continuity and coercivity of $a$ in $K$ we deduce the best approximation property, 
$$\|u^{(h)}-u_h\|_V\le C \inf_{v_h\in V_{g_h}}\|u^{(h)}-v_h\|_V.$$
Taking $v_h=\Pi_h(u)$  and using the triangle inequality we are led to bound
$\|u^{(h)}-u\|_V$ and $\| u-\Pi_h(u)\|_V$. A further use of interpolation estimates allows
to  conclude
\begin{theorem} \label{teo:convergencia_bounded_direct}
Let $\W$ be a bounded, smooth domain, $f\in H^{\regf}(\W)$, $g \in H^{\regg}(\W^c)$ for some $\eps >0$, and assume that $\text{supp}(g) \subset {\W_H}$. 
For the finite element approximations considered in this subsection, it holds that
\[
\| u - u_h \|_{V}  \le C h^{1/2 - \eps} \left(  \| f \|_{H^{\regf}(\W)}  + \| g \|_{H^{\regg}(\W^c)} \right),
\]
for a constant $C$ depending on $\eps$ but independent of $h$, $H$, $f$ and $g$.
 \end{theorem}


\section{Volume constraint truncation error}\label{sec:bdry}
The finite element approximations performed in the previous section refer to a problem in which the Dirichlet condition $g$ has bounded support. Here, we develop error estimates without this restriction on the volume constraints.
However, as it is not possible to mesh the whole support of $g$, we are going to take into account the Dirichlet condition in the set $\W_H$ considered in the previous section. We compare $u$, the solution to \eqref{eq:dirichletnh} to $\tilde u$, the solution to
\begin{equation} \label{eq:dirichlet_tilde}
\left\lbrace
  \begin{array}{rl}
      (-\Delta)^s \tilde u = f & \mbox{ in }\W, \\
      \tilde u = \tilde g & \mbox{ in }\W^c , \\
      \end{array}
    \right.
\end{equation}
where  $\tilde{g} = \eta g$, and $\eta$ is the cutoff function \eqref{eq:cutoff}.
This allows to apply the finite element estimates developed in Section \ref{sec:fe_approximations} to problem \eqref{eq:dirichlet_tilde}, because $\text{supp}(\tilde g) \subset \overline{\W_H}$.
The objective of this section is to show that choosing $H$ in the same fashion as there, namely $H \ge C h^{-1/(n+4s)}$, leads to the same order of error between the continuous truncated problem and the original one.

Since the problems under consideration are linear, without loss of generality we may assume that $g\ge 0$ (otherwise split $g = g_+ - g_-$ and work with the two problems separately). 

\begin{proposition} \label{est_truncado}
The following estimate holds:
\begin{equation}
| u - \tilde u |_{H^s(\W)} \le C H^{-(n/2+2s)} \| g \|_{L^2(\W_H^c)} ,
\label{eq:est_truncado}
\end{equation}
for a constant $C$ independent of $H$ and $g$.
\end{proposition}
\begin{proof}
Denote $\phii = u - \tilde{u}$ the difference between the solutions to equations \eqref{eq:dirichletnh} and \eqref{eq:dirichlet_tilde}.
{Then,
\[ \left\lbrace
  \begin{array}{rll}
      (-\Delta)^s \phii & = 0 & \mbox{ in }\W, \\
      \phii & =  g - \tilde g  \ge 0 & \mbox{ in }\W^c . \\
      \end{array}
    \right.\]
 }
We emphasize that $\phii$ is nonnegative (because of the comparison principle), $s-$harmonic in $\W$ and vanishes in $\W_{H-1}  \setminus \W$. 

Moreover, let us consider $\tilde \phii = \phii \chi_\W$.  As $\phii \in H^{s+1/2-\eps}(\rn)$ vanishes in $\W_{H-1} \setminus \W$, it is clear that $\tilde \phii \in \widetilde H^{s+1/2-\eps}(\W)$, and applying the integration by parts formula \eqref{eq:parts}:
\[
a(\phii, \tilde\phii) = \int_\W \tilde{\phii} (-\Delta)^s \phii = 0.
\]
The nonlocal derivative term in last equation is null because $\tilde \phii$ vanishes in $\W^c$. Splitting the integrand appearing in the form $a$ and recalling the definition of $\ws$ \eqref{eq:def_w}, we obtain
\begin{equation} \label{eq:clave} \begin{split}
|\phii|_{H^s(\W)}^2 & = - 2 \int_\W \phii^2(x) \, \ws(x) \, dx
 + 2 \int_\W \phii(x)  \left( \int_{\W_{H-1}^c} \frac{g(y)-\tilde g (y) }{|x-y|^{n+2s}} \, dy \right) dx \\
& \le 2 \int_\W \phii(x)  \left( \int_{\W_{H-1}^c} \frac{g(y)-\tilde g (y) }{|x-y|^{n+2s}} \, dy \right) dx .
\end{split}
 \end{equation} 
Applying the Cauchy-Schwarz inequality in the integral over $\W_{H-1}^c$ and taking into account that $g - \tilde g \le g$ and that $(H-1)^{-(n/2+2s)} \le C H^{-(n/2+2s)}$, it follows immediately that
\begin{equation} \label{eq:est_hs}
|\phii|_{H^s(\W)}^2 \le C(n,s) H^{-(n/2+2s)}  \| \phii \|_{L^1(\W)} \| g \|_{L^2(\W_{H-1}^c)} .
\end{equation}
We need to bound $\|\phii\|_{L^1(\W)}$ adequately.
Let $\psi \in H^s(\rn)$ be a function that equals $1$ over $\W$. Multiplying $(-\Delta)^s\phii$ by $\psi$, integrating on $\W$ and applying \eqref{eq:parts}, since $\phii$ is $s$-harmonic in $\W$, we obtain
\[
0 = a(\phii, \psi) - \int_{\W^c} \mathcal{N}_s\phi(y) \, \psi(y) \, dy,
\]
or equivalently,
\begin{equation*} 
\begin{split}
0 & = C(n,s) \int_\W \int_{\W^c} \frac{(\phii(x) - \phii(y))(1-\psi(y))}{|x-y|^{n+2s}} \, dy \, dx \\
& - C(n,s) \int_{\W^c} \left(\int_\W \frac{\phii(y)- \phii(x)}{|x-y|^{n+2s}} \, dx \right) \psi(y)  \, dy.
\end{split} 
\end{equation*}
This implies that
\[
\int_\W \int_{\W^c} \frac{\phii(x) - \phii(y)}{|x-y|^{n+2s}} \, dy \, dx = 0 .
\]
Recalling that $\phii$ is zero in $\W_{H-1} \setminus \W$ and that $g - \tilde g \le g$, from the previous identity it follows that
\[
\int_\W \phii(x) \, \ws(x) \, dx =  \int_\W \int_{\W_{H-1}^c} \frac{g(y) - \tilde{g}(y)}{|x-y|^{n+2s}} dy dx \le C H^{-(n/2+2s)}  \| g \|_{L^2(\W_{H-1}^c)} .
\]
Recall that the function $\ws$ is uniformly bounded  in $\W$ and that $\phii\ge 0$. We deduce
\begin{equation} \label{eq:cota_L1}
\| \phii \|_{L^1(\W)} \le C H^{-(n/2+2s)} \| g \|_{L^2(\W_{H-1}^c)} ,
\end{equation}
and combining this bound with \eqref{eq:est_hs} yields \eqref{eq:est_truncado}.
\end{proof}

As a byproduct of the proof of the previous proposition, we obtain the following
\begin{lemma} \label{est_l2} There is a constant $C$ such that the bound
\begin{equation}\label{eq:est_l2}
\| u - \tilde u \|_{L^2(\W)} \le C H^{-(n/2+2s)} \| g \|_{L^2(\W_{H-1}^c)} 
\end{equation}
holds, for a constant $C$ independent of $H$ and $g$.
\end{lemma}
\begin{proof}
As before, we write $\phii = u - \tilde u$.
From the first line in \eqref{eq:clave}, 
\begin{align*}
 2 \int_\W \phii^2(x) \, \ws(x) \, dx
& \le \int_\W \phii(x)  \left( \int_{\W_{H-1}^c} \frac{g(y)-\tilde g (y) }{|x-y|^{n+2s}} \, dy \right) dx \le \\
& \le C H^{-(n/2+2s)} \| \phii \|_{L^1(\W)} \| g \|_{L^2(\W_{H-1}^c)}.
\end{align*}
Combining this estimate with \eqref{eq:cota_L1}, we deduce
\[
\int_\W \phii^2(x) \ws(x) \, dx \le C H^{-(n+4s)} \| g \|_{L^2(\W_{H-1}^c)}^2 ,
\]
where the function $\ws$ is given by Definition \ref{def:w}. 
The lower uniform boundedness of $\ws$ implies \eqref{eq:est_l2} immediately.
\end{proof}

Given $\tilde u$, the solution to \eqref{eq:dirichlet_tilde}, let us denote $\tilde \lambda = \mathcal{N}_s \tilde u$ its nonlocal normal derivative.

\begin{proposition}\label{prop:est_lambda}
There is a constant $C$ such that
\begin{equation} \label{eq:est_lambda}
\| \lambda - \tilde \lambda \|_{\Lambda} \le C H^{-(n/2+2s)}  \| g \|_{L^2(\W_{H-1}^c)} ,
\end{equation}
for a constant $C$ independent of $H$ and $g$.
\end{proposition}
\begin{proof}
Let $\phi \in H^s(\W^c)$, according to Lemma \ref{extension} we consider an extension $E\phi \in H^s(\rn)$ such that $\| E\phi \|_{H^s(\rn)} \le C \| \phi \|_{H^s(\W^c)}.$
By linearity, it is clear that $\lambda - \tilde \lambda = \mathcal{N}_s \phii$, where $\phii = u - \tilde u$. Applying the integration by parts formula \eqref{eq:parts} and recalling that $\phii$ is $s-$harmonic in $\W$,
\[
\int_{\W^c} (\lambda -\tilde \lambda) \phi = \frac{C(n,s)}{2} \iint_Q \frac{(\phii(x)-\phii(y))  (E\phi(x)-E\phi(y))}{|x-y|^{n+2s}} \, dx \, dy.
\]
Since $\phii$ vanishes in $\W_{H-1} \setminus \W$, it is simple to bound
\begin{equation*}\begin{split}
\int_{\W^c} & (\lambda -\tilde \lambda) \phi \le \\
& C \left( \left| \langle \phii, E\phi \rangle_{H^s(\W)} \right| + \left| \int_\W \int_{\W_{H-1}^c}  \frac{(\phii(x)-\phii(y))  (E\phi(x)-\phi(y))}{|x-y|^{n+2s}} \, dx \, dy   \right| \right).
\end{split}\end{equation*}
The first term on the right hand side above is bounded by $C |\phii|_{H^s(\W)} \|\phi\|_{H^s(\W^c)}$, and Proposition \ref{est_truncado} provides the bound $|\phii|_{H^s(\W)} \le C H^{-(n/2+2s)} \| g \|_{L^2(\W_{H-1}^c)}$.
For the second term, splitting the integrand it is simple to obtain the estimates:
\begin{align*}
& \left| \int_\W \phii(x) E\phi(x) \left( \int_{\W_{H-1}^c} \frac{1}{|x-y|^{n+2s}} \, dy \right) dx \right| \le C \| \phii \|_{L^2(\W)} \| E\phi \|_{L^2(\W)}, \\
& \left| \int_\W \phii(x) \left( \int_{\W_{H-1}^c} \frac{\phi(y)}{|x-y|^{n+2s}} \, dy \right) dx \right| \le C H^{-(n/2+2s)} \| \phii \|_{L^1(\W)} \| \phi \|_{L^2(\W_{H-1}^c)}, \\
& \left| \int_\W E\phi(x)\left(\int_{\W_{H-1}^c} \frac{\phii(y)}{|x-y|^{n+2s}} \, dy \right) dx \right| \le  C H^{-(n/2+2s)} \| E \phi \|_{L^1(\W)} \| \phii \|_{L^2(\W_{H-1}^c)}, \\
& \left| \int_\W  \left( \int_{\W_{H-1}^c} \frac{\phii(y) \phi(y)}{|x-y|^{n+2s}} \, dy \right) dx \right| \le C H^{-(n+2s)} \| \phii \|_{L^2(\W_H^c)} \| \phi \|_{L^2(\W_{H-1}^c)}. 
\end{align*}
The terms on the right hand sides of the inequalities above are estimated applying Lemma \ref{est_l2} and Proposition \ref{est_truncado}, as well as recalling the continuity of the extension operator and of the inclusion $L^2(\W) \subset L^1(\W)$. We obtain
\[
\frac{\int_{\W^c} (\lambda -\tilde \lambda) \phi } { \| \phi \|_{H^s(\W^c)} } \le C H^{-(n/2+2s)}  \| g \|_{L^2(\W_{H-1}^c)} \quad \forall \phi \in H^s(\W^c).
\]
Taking supremum in $\phi$, estimate \eqref{eq:est_lambda} follows.
\end{proof}

Combining the estimates obtained in this section, we immediately prove the following result.
\begin{theorem}\label{teo:principal}
Let $(u,\lambda)$ be the solution of problem \eqref{eq:cont}, and consider $\tilde g$ as in the beginning of this section. Moreover, let $(u_h, \lambda_h)$ be the finite element approximations of the truncated problem \eqref{eq:dirichlet_tilde}, defined on $\W_H$, where $H$ behaves as $h^{-1/(n+4s)}$. Then,
\begin{equation}
\| u - u_h \|_{H^s(\W_{H-1})} \le C h^{1/2-\eps}
\, \Sigma_{f,g} \label{eq:aprox_u1}
\end{equation}
and
\begin{equation}
\| \lambda - \lambda_h \|_{\Lambda} \le C h^{1/2-\eps} 
\, \Sigma_{f,g}, \label{eq:aprox_lambda}
\end{equation}
for a constant $C$ depending on $\eps$ but independent of $h$, $H$, $f$ and $g$, and $\Sigma_{f,g}$ defined by \eqref{eq:def_sigma}.
\end{theorem}
\begin{proof}
Applying the triangle inequality, we write
\[
\| u - u_h \|_{H^s(\W_{H-1})} \le \| u - \tilde{u} \|_{H^s(\W_{H-1})} + \| \tilde{u} - u_h \|_{H^s(\W_{H-1})} .
\]
The second term above is bounded by $\| \tilde{u} - u_h \|_V$, which is controlled by \eqref{eq:aprox_u}. As for the first one, recall that $u = \tilde u$ in $\W_{H-1} \setminus \W$, so that 
\begin{equation*}\begin{split}
 \| u - & \tilde{u} \|_{H^s(\W_{H-1})}^2 =  \\
 & \| u - \tilde{u} \|_{H^s(\W)}^2 + 2 \int_\W |u(x) - \tilde{u}(x)|^2 \left( \int_{\W_{H-1}\setminus \W} \frac{1}{|x-y|^{n+2s}} \, dy \right) dx.
\end{split}\end{equation*}
The integral above is bounded by means of Hardy-type inequalities from Proposition \ref{prop:hardy}, (or by Remark \ref{remark:un_medio} if $s=1/2$)
because $(u-\tilde u)\chi_\W$ belongs to $\widetilde H^s(\W)$. So, resorting to Proposition \ref{est_truncado} and Lemma \ref{est_l2},
\[
 \| u - \tilde{u} \|_{H^s(\W_{H-1})} \le  C \| u - \tilde{u} \|_{H^s(\W)} \le C H^{-(n/2+2s)} ,
\]
which --taking into account the behavior of $H$-- is just \eqref{eq:est_truncado} and \eqref{eq:est_l2}.

Estimate \eqref{eq:aprox_lambda} is an immediate consequence of the triangle inequality, the dependence of $H$ on $h$ and equations \eqref{eq:est_lambda} and \eqref{eq:aprox_nsu}. Indeed,
\begin{align*}
\| \lambda - \lambda_h \|_{\Lambda} & \le \| \lambda - \tilde \lambda \|_{\Lambda} + \| \tilde \lambda - \lambda_h \|_{\Lambda} \le 
C h^{1/2-\eps} \, \Sigma_{f,g} .
\end{align*}
\end{proof}

\begin{remark}
We point out that \eqref{eq:aprox_u1} estimates the error in the $H^s(\W_{H-1})$-norm. Since it is only possible to mesh a bounded domain, there is no hope in general to obtain convergence estimates for $\| u - u_h \|_V$, unless some extra hypothesis on the decay of the volume constraint is included.
\end{remark}

\section{Numerical experiments} 
We display the results of the computational experiments performed for the mixed formulation of \eqref{eq:dirichletnh}. The scheme utilized for these two-dimensional examples is based on the code introduced in \cite{ABB}, where details about the computation of the matrix having entries $a(\phii_i, \phii_j)$ can be found.

The examples we provide give evidence of the convergence of the scheme towards the solution $u$ both for Dirichlet data with bounded and unbounded support.  We point out that our convergence estimates (Theorems \ref{teo:convergencia_bounded} and \ref{teo:principal}) are expressed in terms of fractional-order norms, and thus their computation is, in general, out of reach.  Whenever not possible, we compute orders of convergence in $L^2$-norms.

Also, as stated in Remark \ref{rem:s_not_1/2}, although the possibility $s = \frac12$ was excluded from our analysis, the numerical evidence we present here indicates that the same estimates hold in such a case as for $s \ne \frac12$.

Our first example is closely related to Remark \ref{rem:singular}.  Indeed, the solution considered there gives a function with constant fractional Laplacian and supported in the $n$-dimensional unit ball. In this example, however, we shrink the domain so that we produce a nonhomogeneous volume constraint with bounded support. Namely, for $\W = B(0,1/2) \subset \R^2$ we study
\begin{equation}  \label{ex:bounded_support}
\left\lbrace \begin{array}{rl l}
(-\Delta)^s u & = 2 & \text{ in } \W,\\
u & = \frac{1}{2^{2s} \Gamma(1+s)^2} (1 - |\cdot|^2)_+^s & \text{ in } \W^c.
\end{array} \right.
\end{equation} 
By linearity, the exact solution of this problem can be expressed as the sum of the solutions to problems
\begin{equation} \label{eq:auxiliary_1}
\left\lbrace \begin{array}{rl l}
(-\Delta)^s u_1 & = 1 & \text{ in } \W,\\
u_1 & = \frac{1}{2^{2s} \Gamma(1+s)^2} (1 - |\cdot|^2)_+^s & \text{ in } \W^c,
\end{array} \right.
\end{equation}
and
\begin{equation} \label{eq:auxiliary_2}
\left\lbrace \begin{array}{rl l}
(-\Delta)^s u_2 & = 1 & \text{ in } \W,\\
u_2 & = 0 & \text{ in } \W^c.
\end{array} \right.
\end{equation}
The first problem above has a smooth solution within $\overline\W$, whereas the latter has the minimal regularity guaranteed by Proposition \ref{prop:regHr}.
Explicitly, by Remark \ref{rem:singular}, the exact solution is given by
\[
u(x) = u_1(x) + u_2(x) = \frac{1}{2^{2s} \Gamma(1+s)^2} \left[ \left(1 - |x|^2\right)_+^s + \left(\frac14 - |x|^2\right)_+^s \right].
\] 

Moreover, finite element solutions to  \eqref{ex:bounded_support} can also be represented as the sum of the corresponding solutions to \eqref{eq:auxiliary_1} and \eqref{eq:auxiliary_2}.  In practice, we consider the two problems separately and add up their discrete solutions.

The error in the $H^s(\W)$-norm is estimated as follows. In first place, we write
\[
\| u - u_h \|_{H^s(\W)} \le \| u_1 - u_{1,h} \|_{H^s(\W)} +  \| u_2 - u_{2,h}\|_{H^s(\W)}.
\]
As for the first term in the right hand side above, since $u_1$ is smooth in $\W$, we may bound it by interpolation, 
\[
 \|u_1 - u_{1,h} \|_{H^s(\W)} \le  \| u_1 - u_{1,h} \|_{L^2(\W)}^{1-s}  \| u_1 - u_{1,h} \|_{H^1(\W)}^s.
\]
The second term can be computed by using the same trick as in \cite[Lemma 5.1]{AB} because it corresponds to a problem with homogeneous Dirichlet conditions,
\[
 | u_2 - u_{2,h} |_{H^s(\W)} \le  | u_2 - u_{2,h} |_{H^s(\rn)} = \left(\int_\W u_2 (x) - u_{2,h} (x) \right)^{1/2}.
\]

We carried out computations for $s \in \{0.1, \ldots , 0.9\}$ on meshes with size $h \in \{0.045, 0.037, 0.03, 0.025\}$. The auxiliary domains considered were $\W_H = B(0,H+1/2)$ with $H = C h^{-1/(2+4s)}$ and $C = C(s)$ was such that $H$ would equal $1$ if $h$ was set to $0.15$. Therefore, the support of the volume constraint was contained in every auxiliary domain $\W_H$.

Our results are summarized in Tables \ref{tab:bounded_support_s05} and \ref{tab:bounded_support}. In spite of only having upper bounds for the errors, the experimental order of convergence
E.O.C. is in good agreement with the theory. In Table \ref{tab:bounded_support_s05} it is also noticeable that the error is driven by the contribution of the nonsmooth component $u_2$, that is two orders of magnitude larger than the error of the smooth component. The observed order of convergence of the latter is in good agreement with the fact that $u_1 \in H^2(\W)$.

\begin{table}[ht] \centering
\begin{tabular}{| c | c | c | c |} \hline
$h$ & $ \|u_1 - u_{1,h} \|_{H^s(\W)} $ & $ \|u_2 - u_{2,h} \|_{H^s(\W)} $ & $\| u - u_h\|_{H^s(\W)}$ \\ \hline
$0.045$ & $7.593\times10^{-4}$ & $6.423\times10^{-2}$ & $6.499\times10^{-4}$ \\       
$0.037$ & $4.629\times10^{-4}$ & $5.742\times10^{-2}$ & $5.789\times10^{-4}$ \\
$0.030$ & $3.187\times10^{-4}$ & $5.196\times10^{-2}$ & $5.228\times10^{-4}$ \\
$0.025$ & $3.168\times10^{-4}$ & $4.799\times10^{-2}$ & $4.831\times10^{-4}$ \\
\hline \hline
E.O.C. & $1.53$ & $0.49$ & $0.50$ \\
\hline
\end{tabular}
\caption{Upper bounds for the errors in Example \ref{ex:bounded_support} with  $s = 0.5$.
}
\label{tab:bounded_support_s05} 
\end{table}

\begin{table}[ht] \centering
\begin{tabular}{| c | | c | c | c | c | c | c | c | c | c | c |} \hline
$s$ & $0.1$ & $0.2$ & $0.3$ & $0.4$ & $0.5$ & $0.6$ & $0.7$ & $0.8$ & $0.9$ \\ \hline
E.O.C. & $0.48$ & $0.48$ & $0.49$ & $0.49$ & $0.50$ & $0.53$ & $0.56$ & $0.59$ & $0.62$ \\
\hline
\end{tabular}
\caption{Experimental orders of convergence in $H^s(\W)$ for Example \ref{ex:bounded_support}, for $s \in \{0.1, \ldots , 0.9\}$. 
}
\label{tab:bounded_support} 
\end{table}
\medskip

We next display two examples where the Dirichlet condition has unbounded support, posed in the two-dimensional unit ball. The Poisson kernel for this domain is known \cite[Chapter 1]{Landkof}, and thus it is simple to obtain an explicit expression for the solutions of problems as the two we analyze next. More precisely, let $\W = B(0,r) \subset \rn$ for some $r>0$ and let $g:\W^c \to \R$. Then, a solution to 
\begin{equation} \label{eq:balayage} 
\left\lbrace \begin{array}{rl}
(-\Delta)^s u  = 0 & \text{in } \W, \\
u  = g & \text{in } \W^c, 
\end{array}
\right.
\end{equation}
is given by
\begin{equation} \label{eq:sol_balayage}
u(x) = \int_{\W^c} g(y) \, P(x,y) \, dy,
\end{equation}
where 
\[
P(x,y) = \frac{\Gamma(n/2) \sin (\pi s)}{\pi^{n/2 + 1}} \left(\frac{r^2 - |x|^2}{|y|^2 - r^2}\right)^s  \frac1{|x-y|^2}, \ x \in \W, \ y \in \W^c .
\]

We compute numerical solutions to \eqref{eq:balayage} in the two-dimensional unit ball with two different
functions
\[
g(x) = \exp(-|x|^2) \quad  \text{and} \quad  g(x) = \frac1{|x|^4} .
\]

In the experiments performed, we set $\W_H = B(0,H+1)$ with $H = C h^{-1/(2+4s)}$ and $C = C(s)$ such that $H = 1$ for $h=0.1$. We considered discretizations for 
$s \in \{0.1, \ldots , 0.9\}$ on meshes with size $h \in \{0.1, 0.082, 0.067, 0.055, 0.045\}$. 
Table \ref{tab:balayage} shows the computed orders of convergence in $L^2(\W)$ for these two problems, and Figure \ref{fig:fitting_pol} displays the computed $L^2$-errors for some values of $s$ and $g(x) = \frac1{|x|^4}$. In this example solutions are not smooth up to the boundary of $\W$. Thus, the observed convergence with orders approximately $s+1/2$ is expected.   

\begin{table}[ht]\centering
\begin{tabular}{| c | c | c |} \hline
$s$ & $g(x) = \exp(-|x|^2)$ & $g(x) = \frac1{|x|^4}$ \\ \hline
$0.1$ & $0.64$ & $0.55$ \\
$0.2$ & $0.78$ & $0.64$ \\
$0.3$ & $0.86$ & $0.74$ \\
$0.4$ & $0.90$ & $0.89$ \\
$0.5$ & $0.97$ & $1.03$ \\
$0.6$ & $1.15$ & $1.14$ \\
$0.7$ & $1.27$ & $1.16$ \\
$0.8$ & $1.32$ & $1.26$ \\
$0.9$ & $1.37$ & $1.40$ \\
\hline
\end{tabular}
\caption{Experimental orders of convergence in $L^2(\W)$ for \eqref{eq:balayage} with Dirichlet data with unbounded support.}
 \label{tab:balayage} 
\end{table}

\begin{figure}[ht]
	\centering
	\includegraphics[width=0.6\textwidth]{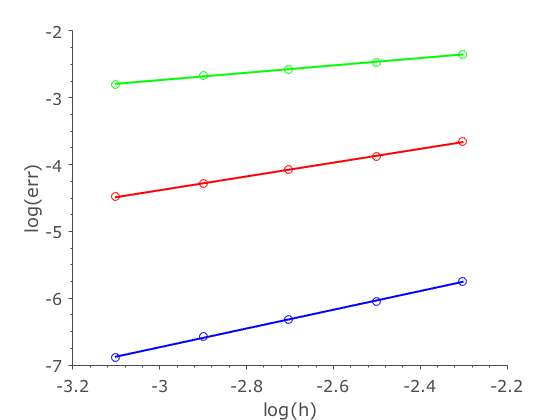} 
	\caption{Computed $L^2$-errors for $s = 0.1$ (green), $s=0.5$ (red) and $s=0.9$ (blue) for problem \eqref{eq:balayage} with $g(x) = \frac1{|x|^4}$.}	\label{fig:fitting_pol}
\end{figure}

Moreover, since we cannot mesh the support of the volume constraint, as $h$ decreases the actual region where we measure the error is expanded. According to Remark \ref{rem:orden_H}, in these experiments we have considered $H = C h^{-1/(4+2s)}$. Nevertheless, the computational cost of solving \eqref{eq:discrete} for $H$ large is extremely high. In practice, we have worked with small values of the constant $C$ that relates $H$ with $h$, especially for $s$ small.  

\medskip

Finally, since it is expected that increasing the truncation parameter $H$ leads to a better approximation, we analyze the dependence on $H$ in the previous example with $g(x) = \frac1{|x|^4}$. We compare convergence rates both in $L^2(\W_H)$ and $L^2(\rn)$. Because
\[ \| g  \|_{L^2(B(0,R)^c)} = \sqrt{\frac{\pi}{3}} \,  R^{-3} , \]
the decay of the error in $\W_H^c$ is algebraic in $h$,
\[
\| g \|_{L^2(\W_H^c)} \le C  h^{\frac{3}{2+4s}}. 
\]

Thus, if we utilize a sequence of domains $\{ \W_H \}$ with $H$ not large enough, the tail of the $L^2$-norm of the volume constraint has a large impact on the $L^2(\rn)$-error.  
In Figure \ref{fig:erres} we compare the effect of increasing the constant in the identity $H = C h^{-1/(2+4s)}.$ Errors are observed to diminish considerably, and there is a slight improvement in the orders of convergence as well. Notice also that the errors in $L^2(\rn)$ are one order of magnitude larger than errors in $L^2(\W_H)$.

\begin{figure}[ht]
	\centering
	\includegraphics[width=0.48\textwidth]{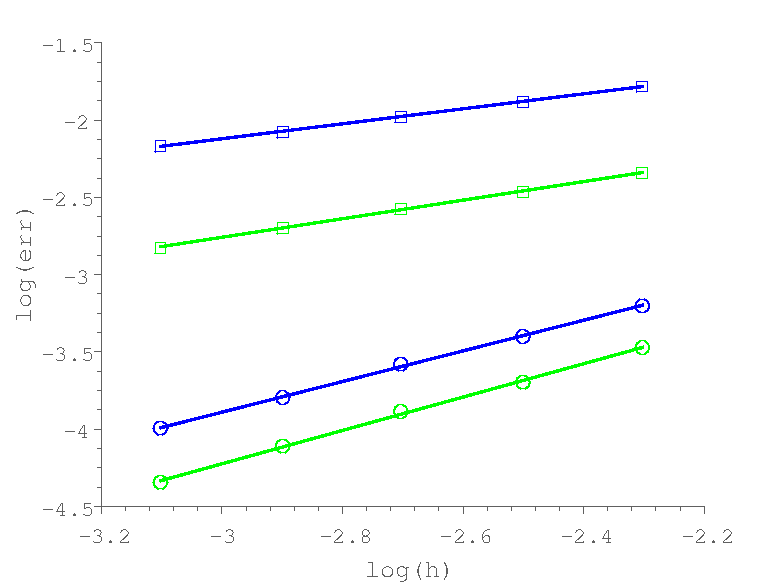} 
	\includegraphics[width=0.48\textwidth]{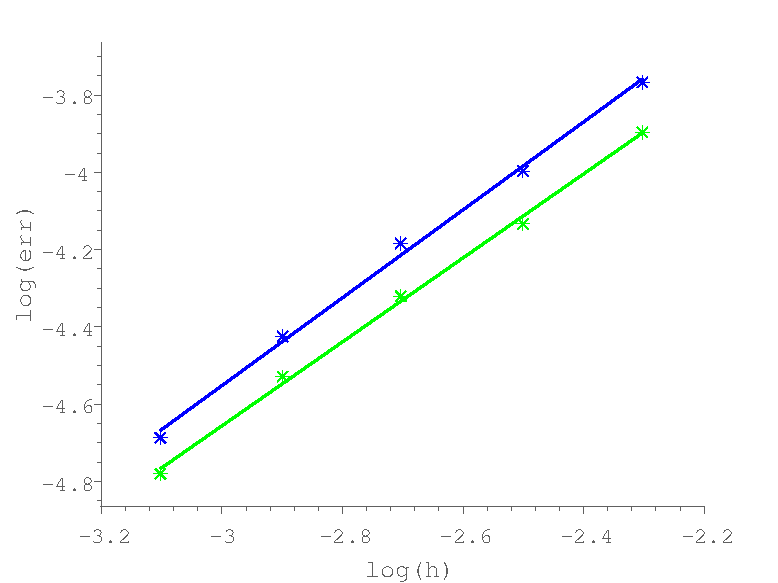} 
\caption{Left panel: convergence in $L^2(\W_H)$ (circles) and in $L^2(\rn)$ squares for problem \eqref{eq:balayage} with $s = 0.6$ and $g(x) = |x|^{-4}$.  Color blue corresponds to a sequence of meshes such that $H = 1$ if $h=0.1$; the slopes of the best fitting lines are $0.99$ and $0.48$, respectively. 
In green, we display the results with a sequence of larger auxiliary domains, corresponding to $H=1.5$ when $h=0.1$. The slopes of the green lines are $1.08$ and $0.60$. 
Right panel: convergence in $L^2(\W)$ for the same problem; the best-fitting lines have slopes $1.14$ (blue) and $1.09$ (green).}	\label{fig:erres}
\end{figure}

\bibliography{abh.bib}{}
\bibliographystyle{plain}

\end{document}